\documentclass[10pt,twoside]{amsart}
\usepackage{amsmath,amssymb,amsthm}
\usepackage{enumitem}
\usepackage[pdftex]{graphicx}
\usepackage{verbatim}
\usepackage{fullpage,relsize,cite,graphicx,caption,subcaption,mathtools}
\usepackage{color}
\newtheorem{lemma}{Lemma}
\newtheorem{corollary}{Corollary}
\newtheorem{definition}{Definition}
\newtheorem{theorem}{Theorem}

\theoremstyle{definition}
\newtheorem{example}{Example}
\newtheorem{remark}{Remark}

\numberwithin{equation}{section}
\numberwithin{lemma}{section}
\numberwithin{corollary}{section}
\numberwithin{theorem}{section}
\numberwithin{proposition}{section}
\numberwithin{problem}{section}
\numberwithin{remark}{section}
\numberwithin{definition}{section}
\numberwithin{assumption}{section}

\begin{document}

\title{Hamilton--Jacobi equations on an evolving surface}
\author{Klaus Deckelnick}
\address{Institut f\"ur Analysis und Numerik, Otto-von-Guericke-Universit\"at Magdeburg,
39106 Magdeburg, Germany}
\email{Klaus.Deckelnick@ovgu.de}
\author{Charles M. Elliott} 
\address{Mathematics Institute, Zeeman Building, University of Warwick, Coventry. CV4 7AL. UK}
\email{C.M.Elliott@warwick.ac.uk}
\author{Tatsu-Hiko Miura}
\address{Graduate School of Mathematical Sciences, The University of Tokyo, 
3-8-1 Komaba, Meguro, Tokyo, 153-8914 Japan}
\email{thmiura@ms.u-tokyo.ac.jp}
\author{Vanessa Styles}
\address{School of Mathematical and Physical Sciences,  University of Sussex, 
Brighton, BN1 9QH}
\email{v.styles@sussex.ac.uk}
\thanks{
The work of CME was partially supported by the Royal Society via a Wolfson Research Merit Award.
The work of THM was partially supported by Grant--in--Aid for JSPS Fellows No. 16J02664
and the Program for Leading Graduate Schools, MEXT, Japan.}

\subjclass[2010]{Primary 65M08, 35F21, 35D40}

\date{}

\begin{abstract} We consider the well-posedness and numerical approximation of a Hamilton--Jacobi 
equation on an evolving hypersurface in $\mathbb R^3$. 
Definitions  of viscosity sub- and supersolutions are extended in a natural way  to evolving hypersurfaces and 
provide uniqueness by comparison. An explicit in time monotone numerical approximation is derived on evolving interpolating triangulated surfaces.
 The scheme relies on a finite volume discretisation which does not require acute triangles. The scheme is shown to be stable and consistent leading to an existence proof via the proof of convergence. Finally an error bound is proved of the same order as in the flat stationary case.
\end{abstract}
\maketitle

%%%%%%%%%%%%%%%%%% SECTION %%%%%%%%%%%%%%%%%%%%%%
\section{Introduction} \label{S:Introduction}
%\subsection{Derivation of the equation} \label{SS:HJ_Derivation}
 It is natural  to study the development of a theory of viscosity solutions and their numerical approximation to first order equations on evolving surfaces which may be useful in the modelling of transport on moving surfaces, for example in material science and cell biology. In this paper we are concerned with the existence, uniqueness and numerical approximation  of Hamilton--Jacobi equations on 
 moving hypersurfaces.  
Let  $\Gamma(t), \, t\in[0,T]$ be a family of smooth, closed, connected and  oriented hypersurfaces in 
$\mathbb{R}^3$ and $S_T:=\bigcup_{t\in(0,T)}\Gamma(t)\times\{t\}$. We consider
 the  following Hamilton--Jacobi equation on the evolving surfaces $\Gamma(t)$
\begin{equation}  \label{HJeqn}
\partial^\bullet u + H(x,t,\nabla_{\Gamma}u)=0 \qquad \mbox{ on } S_T.
\end{equation}
In the above, $\partial^\bullet u= u_t + v_{\Gamma} \cdot \nabla u$ denotes the material derivative, $ v_{\Gamma}$ denotes the  velocity of a parametrisation of $\Gamma$,
and $\nabla_{\Gamma}u =(I_3 - \nu \otimes \nu) \nabla u$ the tangential gradient of $u$,
where $\nu$  is a unit normal field of $\Gamma(t)$
respectively. The precise definitions and assumptions on $H:S_T \times \mathbb{R}^3
\rightarrow \mathbb{R}$ will be given in Sections 2 and 3. 
The well-posedness theory is developed using the concept of viscosity solutions extended to evolving curved hypersurfaces.  Having defined the concept of viscosity solution, uniqueness is proved using comparison and  existence is achieved through proving convergence of  explicit in time finite volume discretisations on evolving triangulations. We prove an error bound which is of the same order as that proved in the seminal work of Crandall and Lions, \cite{CraLio84}, concerning finite difference approximations on flat domains.  In particular we allow for non-acute triangulations of surfaces because in practical computations  initially acute evolving triangulations may lose acuteness.

\subsection{Background}
 Partial differential equations on time evolving hypersurfaces arise in many applications in biology, fluids and materials science,  
 for example see \cite{EllStiVen12, BarGarNur15, JanOlsReu17, EilEll08, BreDuEll16} and the references cited therein. The theory of parabolic equations has been 
 considered in \cite{DziEll07-a, Vie14, AlpEllSti15a, AlpEllSti15}. Existence and uniqueness  of first order  scalar conservation laws on  
 moving hypersurfaces and Riemannian manifolds  
has been proved in \cite{DziKroMul13, LenMul13}.  Viscosity solutions of Hamilton--Jacobi equations on Riemannian manifolds are considered in \cite{ManMen03}. See \cite{CheBurMer02} and \cite{MacRuu08} for  level set approaches to the motion of curves  on a stationary surface. Numerical transport on evolving surfaces by level set methods was considered in \cite{AdaSet03, XuZha03}. The numerical analysis of advection diffusion equations on evolving surfaces via the evolving surface finite element method began  in \cite{DziEll07-a}, see also \cite{DE13Acta, KovPow15}. Finite volume schemes for diffusion and conservation laws on moving surfaces have been 
considered, respectively,  in \cite{LenNemRum11} and \cite{GM14NM}. Other approaches  involve  diffuse interfaces, see \cite{TeiLiLow09}, or trace finite elements, \cite{OlsReu14}.

\subsection{An  example} \label{s:motivating}
  One motivation for considering Hamilton--Jacobi equations of the form (\ref{HJeqn})  is to consider the motion of curves on an evolving surface.  Consider the motion of a closed curve $\gamma(t) \subset
\Gamma(t)$ according to the evolution law
\begin{equation}  \label{evlaw}
V_\mu(x,t) = F(x,t) +\beta (x,t)\cdot \mu (x,t), \qquad x \in  \gamma(t),
\end{equation}
where $V_\mu$ denotes the velocity of $\gamma(t)$ in the direction of the conormal $\mu$
and $F:S_T \rightarrow \mathbb{R}, ~\mbox{and}~ \beta : S_T \rightarrow \mathbb{R}^3$ are a given function and vector field.  Let us assume that 
\begin{displaymath}
\gamma(t) = \{(x,t)\in S_T \mid u(x,t)=r\}
\end{displaymath}
for some $r \in \mathbb{R}$ with a function $u:N_T \rightarrow \mathbb{R}$  satisfying $\nabla_\Gamma u(\cdot,t)\neq 0$ on $\gamma(t)$,
where $N_T$ is an open neighbourhood of $S_T$.   Choosing  parametrizations
$\varphi(\cdot,t): S^1 \rightarrow \mathbb{R}^3$ of $\gamma(t)$ we have that
$u(\varphi(s,t),t)=r$ for $s \in S^1, t \in (0,T)$. If we differentiate  both sides with respect to $t$ we obtain
\begin{align*}
u_t(\varphi(s,t),t)+\varphi_t(s,t)\cdot\nabla u(\varphi(s,t),t) = 0,
\end{align*}
or equivalently
\begin{eqnarray*}
0 & = & \partial^\bullet u(\varphi(s,t),t) + \bigl( \varphi_t(s,t)- v_{\Gamma}(\varphi_t(s,t),t) \bigr) \cdot
\nabla u(\varphi(s,t),t) \\
& = &  \partial^\bullet u(\varphi(s,t),t) + \bigl( \varphi_t(s,t)- v_{\Gamma}(\varphi_t(s,t),t) \bigr) \cdot
\nabla_{\Gamma}  u(\varphi(s,t),t),
\end{eqnarray*}
since $\varphi(s,t) \in \Gamma(t)$ implies that $(\varphi_t(s,t) - v_{\Gamma}(\varphi(s,t),t))
\cdot \nu(\varphi(s,t),t)=0$. Using that $\mu=\frac{\nabla_{\Gamma}u}{| \nabla_{\Gamma} u |}$
we obtain from (\ref{evlaw}) that at $x=\varphi(s,t)$
\begin{displaymath}
F(x,t)+\beta (x,t) \cdot \mu (x,t)= V_{\mu}(x,t)= \varphi_t(s,t) \cdot \frac{\nabla_{\Gamma}u(x,t)}{| \nabla_{\Gamma} u(x,t) |}
= - \frac{\partial^\bullet u(x,t)}{| \nabla_{\Gamma} u(x,t) |} + v_{\Gamma}(x,t) \cdot
\frac{\nabla_{\Gamma}u(x,t)}{| \nabla_{\Gamma} u(x,t) |}.
\end{displaymath}
Formally the above calculations then show that the level sets of a solution $u$ of (\ref{HJeqn})
with
\begin{equation}   \label{E:Intro_Hamiltonian_Form}
H(x,t,p)= F(x,t) \, | p | +\beta(x,t) \cdot p- v_{\Gamma}(x,t) \cdot p
\end{equation}
evolve according to the evolution law (\ref{evlaw}). Model examples of curve evolution on a given  moving surface are presented in Section 7.

%%%%%%%%%%%%%%%%%%%%%% SUBSECTION %%%%%%%%%%%%%%%%%%%%%%%
\subsection{Outline}
The paper is organized as follows. We begin in Section 2 by establishing some notation and concepts relating to moving surfaces. 
In Section 3 we generalise the classical definition of viscosity solution  (see e.g.  \cite{BC97book, G06book, Bar13}) to moving curved 
domains using surface derivative operators instead of the usual
derivatives. In this setting we show  that a comparison principle holds which yields uniqueness of a 
viscosity solution.  As in the work \cite{CraLio84} we approach existence via a discretisation in space and time. To do so, 
we approximate the moving surfaces by triangulated surfaces so that we need to
formulate our numerical scheme on unstructured meshes.
Numerical schemes for Hamilton--Jacobi equations on unstructured meshes on flat domains have been proposed in
  \cite{KMS99} and  \cite{LiYanChan03}. In order to guarantee monotonicity of their schemes the
  authors in  \cite{KMS99},  \cite{LiYanChan03} have to assume that the underlying triangulation is acute, which is a rather strong requirement and
 difficult to realise  in the case of moving surfaces where the triangulation will vary from time step to time step.
In order to address this issue  we construct  in Section 4 a finite volume scheme by adapting  an idea 
introduced by Kim and Li in \cite{KL15JCM} 
to the case of evolving hypersurfaces. With this construction which allows non-acute triangles we are able to prove monotonicity and
consistency  assuming only regularity of the  triangulation.
In Section 5 we prove that the sequence of discrete
solutions obtained via our scheme converges to a viscosity solution  if the discretization parameters
tend to zero. At  the same time this gives 
an existence result for  the Hamilton--Jacobi equation. We prove in Section 6 an $O(\sqrt{h})$
error bound between the viscosity solution and the numerical solution extending well--known 
error estimates for the flat case to the case of moving hypersurfaces. Finally in Section 7 we present some model numerical examples and discuss numerical issues.

\section{Preliminaries} \label{S:Preliminaries}

%%%%%%%%%%%%%%%%%%%%%%%%%%SUBSECTION %%%%%%%%%%%%
\subsection{Tangential derivatives of functions on fixed surfaces} \label{SS:Tan_Grad_Fixed}
Let $\Gamma$ be a smooth, closed (i.e. compact without boundary) and orientable
hypersurface in $\mathbb{R}^3$ with outward unit normal field $\nu$.  For a differentiable function $f$ on $\Gamma$ 
we define the tangential gradient by 
  \begin{align} \label{E:Tangential_Gradient}
    \nabla_\Gamma f(x) := P_\Gamma(x)\nabla\tilde{f}(x), \quad x\in \Gamma,
  \end{align}
where $\tilde{f}$ is a smooth extension of $f$ to an open neighbourhood $N$  of $\Gamma$
satisfying $\tilde{f}=f$ on $N \cap \Gamma$ and $P_\Gamma(x):=I_3-\nu(x)\otimes\nu(x)$ is the orthogonal projection 
onto the tangent plane of $\Gamma$ at $x$.  Here $I_3$ is the $(3 \times 3)$ identity matrix and $\nu\otimes\nu=(\nu_i\nu_j)_{i,j}$ 
where $\otimes$ denotes  the tensor product. It is well--known that $\nabla_{\Gamma}f(x)$
is independent of the particular extension $\tilde{f}$. 
% We shall write $\underline{D}_i f(x)$ for the $i$-th component of $\nabla_\Gamma f(x)$. 
Furthermore, we define by
$\Delta_{\Gamma} f:= \nabla_{\Gamma} \cdot \nabla_{\Gamma} f $
%= \sum_{i=1}^3 \underline{D}_i \underline{D}_i f$ 
the Laplace--Beltrami operator of $f$. We denote by $d$ the signed distance 
function to $\Gamma$ oriented in such a way that it increases in the direction of $\nu$. There
exists an open neighbourhood $U$ of $\Gamma$ such that $d$ is smooth in $U$
and such that for every $x \in U$ there exists a unique $\pi(x) \in \Gamma$ with
\begin{equation}  \label{fermi}
x = \pi(x) + d(x) \nu(\pi(x)) \quad  \mbox{ and }\quad 
\nabla d(x)  =  \nu(\pi(x)).
\end{equation}
For a given function $f:\Gamma \rightarrow \mathbb{R}$ we can define  $f_c:U
\rightarrow \mathbb{R}$ via $f_c(x):=f(\pi(x))$, which extends $f$ constantly in the normal direction to
$\Gamma$. It is not difficult to verify that 
\begin{align}
   \nabla f_c(x) &= \nabla_{\Gamma} f(x), \quad x \in \Gamma, \label{E:Ext_f} \\
    \|\nabla f_c\|_{B(U)} &\leq c\|\nabla_\Gamma f\|_{B(\Gamma)}, \label{E:Estimate_Nabla_Ext}\\
    \|\nabla^2f_c\|_{B(U)} &\leq c\left(\|\nabla_\Gamma f\|_{B(\Gamma)}+\|\nabla_\Gamma^2f\|_{B(\Gamma)}\right), \label{E:Estimate_Hess_Ext}
\end{align}
provided that the derivatives of $f$ exist. Here, $\Vert f \Vert_{B(D)}:= \sup_{x \in D}| f(x)|$. \\[3mm]

\subsection{Time dependent surfaces}
Let us next turn to the case of  time dependent  surfaces and assume that $\Gamma_0$
is a closed, connected, oriented and smooth hypersurface in $\mathbb{R}^3$. We consider
a family $\{\Gamma(t)\}_{t\in[0,T]}$, $T>0$ of evolving hypersurfaces given via a smooth  flow map
$\Phi: \Gamma_0 \times [0,T] \rightarrow \mathbb{R}^3$ such that $\Phi(\cdot,t)$ is a
diffeomorphism of $\Gamma_0$ onto $\Gamma(t)$ satisfying
\begin{align} \label{E:Flow_Map}
\frac{\partial\Phi}{\partial t}(X,t) = v_\Gamma(\Phi(X,t),t), \quad   \Phi(X,0) = X, 
\end{align}
for all $X\in\Gamma_0, t \in  (0,T)$. Here we say that  $v_{\Gamma}$ is the velocity field of $\Gamma(t)$.
%It can be decomposed as $v_\Gamma=V_N \nu +v_T$, where $\nu(\cdot,t)$ is the unit outward 
%normal of $\Gamma(t)$, $V_N$ the corresponding normal velocity  and $v_T$  a given tangential velocity field. 
Let $d(\cdot,t)$ be the signed distance function to $\Gamma(t)$ increasing in the direction of $\nu(\cdot,t)$,
where $\nu(\cdot,t)$ is the unit outward normal of $\Gamma(t)$. 
For each $t\in[0,T]$ there exists a bounded open subset $N(t) \subset \mathbb{R}^3$ such that
$d$ is smooth in  $N_T:=\bigcup_{t\in(0,T)}(N(t)\times\{t\})$ and such that for every $x \in N(t)$ there
exists a unique $\pi(x,t) \in \Gamma(t)$ satisfying (\ref{fermi}). 

\begin{comment}
$v_\Gamma$, the inverse of the flow map $\Phi^{-1}$, and the orthogonal projection $P_\Gamma$ are smooth on $\overline{S_T}$.
Then, since $\overline{S_T}$ is compact in in $\mathbb{R}^4$, they are bounded and Lipschitz continuous on $\overline{S_T}$ in the sense that
\begin{align}
  |v_\Gamma(x,t)-v_\Gamma(y,s)| &\leq L_V(|x-y|+|t-s|), \label{E:Lipschitz_Velocity} \\
  |\Phi^{-1}(x,t)-\Phi^{-1}(y,s)| &\leq L_{IF}(|x-y|+|t-s|), \label{E:Lipschitz_Flow_Inv} \\
  |P_\Gamma(x,t)-P_\Gamma(y,s)| &\leq L_P(|x-y|+|t-s|) \label{E:Lipschitz_Orth_Proj}
\end{align}
for all $(x,t),(y,s)\in\overline{S_T}$, where $L_V$, $L_{IF}$, and $L_P$ are some positive constants independent of $(x,t)$ and $(y,s)$.
Note that we use the norm of the ambient Euclidean space, not the Riemannian distance of $\overline{S_T}$ given by the 
infimum of the length of curves on $\overline{S_T}$ connecting $(x,t)$ and $(y,s)$.
\end{comment}
\noindent
Next, for a differentiable function $f$ on $S_T$, the material derivative of $f$ along the velocity $v_\Gamma$ is defined as
  \begin{align*}
    \partial^\bullet f(\Phi(X,t),t) = \frac{d}{dt}\Bigl(f(\Phi(X,t),t)\Bigr), \quad (X,t)\in\Gamma_0 \times(0,T).
  \end{align*}
%We also write $\dot{f}$ for $\partial^\bullet f$.
The material derivative is also expressed as
\begin{align} \label{E:Material_Euler}
  \partial^\bullet f(x,t) = \partial_t\tilde{f}(x,t)+v_\Gamma(x,t)\cdot\nabla\tilde{f}(x,t), \quad (x,t)\in S_T,
\end{align}
where $\tilde{f}$ is an arbitrary extension of $f$ to $N_T$ satisfying $\tilde{f}|_{S_T}=f$. \\

\subsection{Triangulated surface}
In order to approximate the evolving surfaces $\Gamma(t)$ we choose a family of triangulations 
 $(\mathcal{T}_h(t))_{0<h <h_0}$ of $\Gamma(t)$ and set
\begin{displaymath}
\Gamma_h(t):= \bigcup_{K(t) \in \mathcal{T}_h(t)} K(t) \quad \mbox{ and } \quad  h:=\max_{t \in [0,T]}
\max_{K(t) \in \mathcal T_h(t)} h_{K(t)},
\end{displaymath}
where $h_{K(t)}=\mbox{diam}K(t)$ for each triangle $K(t)$. We assume  that the vertices  of the triangulation are advected with the velocity $v_\Gamma$ and thus the number 
of the vertices, which we refer to as $M\in\mathbb{N}$, is fixed in time.
For $i=1,\dots,M$ we call the $i$-th vertex simply $i$ and write $x_i^0\in\Gamma(0)$ for its point at $t=0$.
By the assumption on the motion of the vertices, the position of $i$ at time $t\in[0,T]$ is given by $x_i(t)=\Phi(x_i^0,t)\in\Gamma(t)$ so that the 
triangulated surfaces $\Gamma_h(t)$  are  interpolations of $\Gamma(t)$. %Denoting by $(x_i(t))_{i=1}^M$
%the vertices of the triangulation $\mathcal{T}_h(t)$ we suppose that $x_i(t)=\Phi(x_i^0,t) \in \Gamma(t)$,
%where $x_i^0 \in \Gamma(0), i=1,\ldots,M$. 
In particular, $\Gamma_h(t) \subset N(t)$ if $h_0$ is
sufficiently small and we assume that $\pi_h(\cdot,t):=\pi(\cdot,t)|_{\Gamma_h(t)}$ is a homeomorphism
of $\Gamma_h(t)$ onto $\Gamma(t)$ for each $t \in [0,T]$.
Writing $\pi_h^{-1}(\cdot,t)$ for the inverse map, we define the lift  of a function $\eta:\Gamma_h(t)
\rightarrow \mathbb{R}$ onto $\Gamma(t)$ by
\begin{align*}
  \eta^l(x) := \eta(\pi_h^{-1}(x,t)), \quad x \in \Gamma(t).
\end{align*}

\noindent
We assume that the triangulations $\mathcal T_h(t)$ are regular in the sense that there exists a constant
$\gamma>0$ such that
\begin{align} \label{E:Quotient_Diam_Local}
 \forall t \in [0,T] \; \forall K(t) \in \mathcal T_h(t) \qquad h_{K(t)} \leq \gamma \rho_{K(t)},
\end{align}
where $\rho_{K(t)}$ is the radius of the largest circle contained in  $K(t)$. The existence of $\gamma$ 
follows from the Lipschitz continuity of $\Phi(\cdot,t)$ and $\Phi(\cdot,t)^{-1}$ if we suppose that
the initial triangulation is regular.
We denote by $\nu_h(\cdot,t)$ the
unit normal to $\Gamma_h(t)$ oriented in the direction in which the signed distance $d(\cdot,t)$ increases.
It is well--known that for all $K(t) \subset \Gamma_h(t)$, (c.f. \cite {DziEll07-a,DE13Acta}),
\begin{eqnarray}
\Vert d(\cdot,t) \Vert_{B(K(t))} & \leq & C h_{K(t)}^2, \label{dest} \\
\Vert \nu_{h|K(t)} -  \nu(\cdot,t) \Vert_{B(K(t))} & \leq & C h_{K(t)},  \label{nuest}
\end{eqnarray}
where we can think of $\nu(\cdot,t)$ as being extended to a neighbourhood of $\Gamma(t)$ via
$\nu(x,t)=\nabla d(x,t)$ (cf. (\ref{fermi})). \\
\noindent
For each $t\in[0,T]$ we introduce the finite element space
 \begin{displaymath}
  V_h(t) = \{u_h\in C^0(\Gamma_h(t)) \mid \text{$u_h|_{K(t)}$ is linear affine for each $K(t)\in\mathcal{T}_h(t)$}\}
\end{displaymath}
together with its standard nodal basis $\chi_1(\cdot,t),\ldots,\chi_M(\cdot,t)$, where 
$\chi_i(\cdot,t) \in V_h(t)$  satisfies $\chi_i(x_j(t),t)=\delta_{ij}$.

%%%%%%%%%%
%%%%%%%%%%
\begin{comment}
as well as the corresponding lifted space
\begin{displaymath}
  V_h^l(t) = \{v_h \in C^0(\Gamma(t)) \mid \text{$v_h=u_h^l$ for some $u_h\in V_h(t)$}\}.
\end{displaymath}
By $\chi_1,\dots,\chi_M$ we denote the nodal basis of $V_h$, i.e. for each $i=1,\dots,M$ and $t\in[0,T]$ the function $\chi_i(\cdot,t)$ is in $V_h(t)$ and satisfies $\chi_i(x_j(t),t)=\delta_{ij}$.
Also, for each $i=1,\dots,M$ we define the lift of the nodal basis
\begin{align*}
  \chi_i^l(x,t) := \chi_i(\pi_h^{-1}(x,t),t), \quad x\in\Gamma(t),\,t\in[0,T].
\end{align*}
Finite elements $u_h\in V_h(t)$ and $v_h\in V_h^l(t)$ are expressed as
\begin{align*}
  u_h(\tilde{x}) = \sum_{i=1}^Mu_i\chi_i(\tilde{x},t),\quad v_h(x) = \sum_{i=1}^Mv_i\chi_i^l(x,t)
\end{align*}
for $\tilde{x}\in\Gamma_h(t)$ and $x\in\Gamma(t)$, where $u_i$, $v_i\in\mathbb{R}$, $i=1,\dots,M$.
Note that $v_h=u_h^l$ if and only if $v_i=u_i$ for all $i=1,\dots,M$.
\end{comment}
\noindent
For a function $\eta \in C^0(\Gamma(t))$ we define  the linear interpolation $I_h^t \eta \in V_h(t)$  by
\begin{align*}
  I_h^t \eta (x) := \sum_{i=1}^M \eta(x_i(t))\chi_i(x,t), \quad x\in\Gamma_h(t).
\end{align*}
\begin{comment}
For a function $\eta$ on $\overline{S_T}$ and $t\in[0,T]$ we simply write $I_h^t \eta$ for the interpolation of $\eta(\cdot,t)$ on $\Gamma_h(t)$.
The lift of the interpolation $I_h^t \eta$ is of the form
\begin{align*}
  [I_h^t \eta]^l(x) = \sum_{i=1}^M \eta(x_i(t))\chi_i^l(x,t), \quad (x,t)\in\overline{S_T}.
\end{align*}
\end{comment}
%%%%%%%%%%
%%%%%%%%%%

\begin{lemma} \label{L:Interpolation}
  Suppose that $\eta\colon\Gamma(t)\to\mathbb{R}$, $t\in[0,T]$ is Lipschitz continuous, i.e. there exists a constant $L_U>0$ such that
  \begin{align} \label{E:Inter_Lip_U}
    | \eta(x)- \eta(y)| \leq L_U|x-y|, \quad x,y\in\Gamma(t).
  \end{align}
  Then we have
  \begin{align} \label{E:Sup_Diff_Inter}
%    \|u^{-l}-I_h^tu\|_{B(\Gamma_h(t))} \leq c_dL_Uh, \quad
     \|  \eta -[I_h^t \eta ]^l\|_{B(\Gamma(t))} \leq Ch.
  \end{align}
\end{lemma}

\begin{proof} Fix $x \in \Gamma(t)$. Then there exists $\tilde x  \in \Gamma_h(t)$ such that
 $x=\pi_h(\tilde x,t)$, say $\tilde x \in K(t)$ for some $K(t) \in \mathcal T_h(t)$. Assuming for
 simplicity that the vertices of $K(t)$ are $x_1(t),x_2(t)$ and $x_3(t)$ we may write
 \begin{displaymath}
 \eta(x) - [ I_h^t \eta]^l(x)= \eta(x) - \sum_{i=1}^3 \eta(x_i(t)) \chi_i(\tilde x,t)
 = \sum_{i=1}^3 \bigl( \eta(x) - \eta(x_i(t)) \bigr) \chi_i(\tilde x,t),
 \end{displaymath}
 since $\sum_{i=1}^3 \chi_i(\tilde x,t)=1$. Combining this relation with the fact that $\chi_i(\tilde x,t) \geq 0$,
 (\ref{E:Inter_Lip_U}), (\ref{fermi}) and (\ref{dest})  we deduce that
 \begin{eqnarray*}
 | \eta(x) -  [ I_h^t \eta]^l(x) | & \leq & \max_{i=1,2,3} | \eta(x) - \eta(x_i(t)) | \leq L_U \max_{i=1,2,3} | x - x_i(t) |  
  =  L_U  \max_{i=1,2,3}  | \pi(\tilde x,t) - x_i(t) |  \\
  & \leq & L_U \max_{i=1,2,3} \bigl( | \tilde x - x_i(t) | +| d(\tilde x,t)| \bigr) 
  \leq  L_U \bigl( h_{K(t)} + C h_{K(t)}^2 \bigr) \leq C h_{K(t)} \leq Ch.
 \end{eqnarray*}
 \end{proof}

\section{Viscosity solutions: Uniqueness} \label{S:Viscosity_Solutions}

%\subsection{Definition of viscosity solutions} \label{SS:Def_Viscosity}
We consider the Hamilton--Jacobi equation
\begin{align} \label{E:HJ_Equation}
  \begin{cases}
    \partial^\bullet u(x,t)+H(x,t,\nabla_\Gamma u(x,t)) = 0, &(x,t)\in S_T, \\
    u(x,0) = u_0(x), &x\in\Gamma(0).
  \end{cases}
\end{align}
Here $H\colon \overline{S_T}\times\mathbb{R}^3\to\mathbb{R}$ is a Hamiltonian and $u_0\colon\Gamma(0)\to\mathbb{R}$ is an initial value.
Throughout this paper we suppose that  $u_0\in C(\Gamma(0))$ and there exist positive constants $L_{H,1}$ and $L_{H,2}$ such that
\begin{align} 
  |H(x,t,p)-H(y,s,p)| & \leq L_{H,1}(|x-y|+|t-s|)(1+|p|), \label{E:Hamiltonian_Lip_xt}  \\
  |H(x,t,p)-H(x,t,q)| &  \leq L_{H,2}|p-q|  \label{E:Hamiltonian_Lip_p}
\end{align}
for all $(x,t), (y,s)\in\overline{S_T}$ and $p,q\in\mathbb{R}^3$. Furthermore, we assume for the
velocity field that $v_{\Gamma} \in C^1(\overline{S_T})$. Note that the Hamiltonian in  (\ref{E:Intro_Hamiltonian_Form}) 
satisfies the above conditions provided that $F$ and $ \beta$ 
are Lipschitz on $\overline{S_T}$. \\
For $\Gamma=\Gamma(t)$ with each fixed $t\in[0,T]$ or $\Gamma=\overline{S_T}$, we
denote by $USC(\Gamma) $ (resp. $LSC(\Gamma)$) the set of all upper (resp. lower)
semicontinuous functions on $\Gamma$.  In what follows we shall work in the framework of discontinuous viscosity solutions.

\begin{comment}
For a locally bounded function $u$ on $\overline{S_T}$, its upper semicontinuous envelope $u^\ast$ and lower semicontinuous envelope $u_\ast$ are defined as
  \begin{align*}
    u^\ast(x,t) := \limsup_{\overline{S_T}\ni (y,s)\to(x,t)}u(y,s), \quad u_\ast(x,t) := \liminf_{\overline{S_T}\ni (y,s)\to(x,t)}u(y,s)
  \end{align*}
 for $(x,t)\in\overline{S_T}$. It is well--known (see e.g. ~\cite[Section~V.2.1, Proposition~2.1]{BC97book})
 that $u^\ast\in USC(\overline{S_T}), \, u_\ast\in LSC(\overline{S_T})$, 
 $u_\ast\leq u\leq u^\ast$ on $\overline{S_T}$ and that 
  $u=u^\ast$ (resp. $u=u_\ast$) on $\overline{S_T}$ if and only if $u\in USC(\overline{S_T})$ 
  (resp. $u\in LSC(\overline{S_T})$). 
\end{comment}

\begin{definition} \label{D:Def_Viscosity_Solutions}
Let $u_0$ be a function on $\Gamma(0)$. A locally bounded function $u \in USC(\overline{S_T})$ 
(resp. $u \in LSC(\overline{S_T})$) is 
called a viscosity subsolution (resp. supersolution) of \eqref{E:HJ_Equation} if $u(x,0)\leq u_0(x)$ 
(resp. $u(x,0) \geq u_0(x)$) for all $x\in\Gamma(0)$ and, for any $\varphi\in C^1(\overline{S_T})$, if $u-\varphi$ takes a local maximum (resp. minimum) at $(x_0,t_0)\in\overline{S_T}$ with $t_0>0$, then
    \begin{align} \label{E:Ineq_Subsol}
      \partial^\bullet\varphi(x_0,t_0)+H(x_0,t_0,\nabla_\Gamma\varphi(x_0,t_0)) \leq 0 \quad (\mbox{resp. }
      \geq 0).
\end{align}
If $u$ is a sub- and supersolution, then we call $u$ a viscosity solution to \eqref{E:HJ_Equation}.
\end{definition}
\noindent
By the definition above, a viscosity solution is continuous and satisfies $u(x,0)=u_0(x), \, x \in \Gamma(0)$.
In Section~\ref{S:Convergence} we prove that the upper and lower weak limits of a sequence of approximate solutions are a subsolution and supersolution, respectively, and then obtain a viscosity solution by showing that the upper weak limit agrees with the lower weak limit.
For this argument and the uniqueness of a viscosity solution the following comparison principle is crucial.

%Hereafter we call a viscosity subsolution and supersolution just a subsolution and supersolution, respectively.

\begin{comment}
\begin{lemma} \label{L:Visc_Sol_Conti}
  Suppose that $u$ is a viscosity solution to \eqref{E:HJ_Equation} with initial value $u_0$.
  Then $u\in C(\overline{S_T})$ and $u(\cdot,0)=u_0$ on $\Gamma(0)$.
  In particular, it is necessary for the existence of a viscosity solution $u$ that $u_0\in C(\Gamma(0))$.
\end{lemma}

\begin{proof}
  By the definition of a subsolution and supersolution and Lemma~\ref{L:Property_Envelope} (3) and (4), if $u$ is a viscosity solution to \eqref{E:HJ_Equation}, then $u^\ast$ and $u_\ast$ are, respectively, a subsolution and supersolution satisfying
  \begin{align} \label{Pf_VSC:Initial}
    u^\ast(\cdot,0)\leq u_0\leq u_\ast(\cdot,0) \quad\text{on}\quad \Gamma(0).
  \end{align}
  We apply the comparison principle (see Lemma~\ref{L:Comparison_Principle} below) to $u^\ast$ and $u_\ast$ to get $u^\ast\leq u_\ast$ on $\overline{S_T}$.
  This inequality together with Lemma~\ref{L:Property_Envelope} (1) yields $u=u^\ast=u_\ast$ on $\overline{S_T}$ and, in particular, $u(\cdot,0)=u_0$ on $\Gamma(0)$ by \eqref{Pf_VSC:Initial}.
  Since $u^\ast\in USC(\overline{S_T})$ and $u_\ast\in LSC(\overline{S_T})$ by Lemma~\ref{L:Property_Envelope} (2), we conclude that $u\in C(\overline{S_T})$.
\end{proof}
\end{comment}

%\subsection{Comparison principle} \label{SS:Comparison_Principle}

\begin{theorem} \label{L:Comparison_Principle}
  Let $u$  be a subsolution and $v$ be a supersolution of \eqref{E:HJ_Equation}.
  Suppose that $u(\cdot,0)\leq v(\cdot,0)$  on $\Gamma(0)$.
  Then  $u\leq v$ on $\overline{S_T}$.
\end{theorem}
\begin{proof} We essentially use a standard argument that  is e.g. outlined in \cite[Section 5]{Bar13}.
Let us define for $\eta>0$ the function $u_{\eta}(x,t):=u(x,t)-\eta t$. Clearly, 
$u_{\eta} \in USC(\overline{S_T})$ and $u_{\eta}(\cdot,0) \leq v(\cdot,0)$ on $\Gamma(0)$. Since  $ v \in LSC(\overline{S_T})$ we have $u_{\eta}-v \in USC(\overline{S_T})$
so that $\sigma_{\eta}: = \max_{\overline{S_T}}(u_{\eta}-v)$ exists.
Let us suppose  that $\sigma_{\eta}>0$. We use the 
  doubling of variables technique and  define for 
  $0 < \alpha \ll 1$ 
  \begin{align*}
    \Psi_{\alpha}(x,t,y,s) := u_{\eta}(x,t)-v(y,s)-\frac{|x-y|^2+| t-s|^2}{\alpha^2}, \quad
    (x,t,y,s)\in\overline{S_T}\times\overline{S_T}.
  \end{align*}
 $\Psi_{\alpha}$ is upper semicontinuous on $\overline{S_T}\times\overline{S_T}$ and
 hence attains a maximum at some point $(\bar x, \bar t,\bar y,\bar s)\in\overline{S_T}\times\overline{S_T}$,
 where we suppress the dependence on $\alpha$. It is shown in \cite[Lemma 5.2]{Bar13}
 that 
 \begin{eqnarray}
 \frac{| \bar x - \bar y |^2}{\alpha^2}, \; \frac{ | \bar t - \bar s |^2}{\alpha^2}  \rightarrow  0, &&
 \mbox{ as } \alpha \rightarrow 0, \label{conv1} \\
 \bar t, \bar s >0, && \mbox{ for small  }  \alpha>0.  \label{conv2}
 \end{eqnarray}

\noindent
  We define for $(x,t)$, $(y,s)\in\mathbb{R}^4$ the functions
  \begin{align*}
    \varphi^1(x,t) := v(\bar y, \bar s)+\frac{|x- \bar y |^2+| t- \bar s|^2}{\alpha^2}, \;
    \varphi^2(y,s) := u_{\eta}(\bar x, \bar t) -\frac{| \bar x-y |^2+ | \bar t - s|^2}{\alpha^2}.
  \end{align*}
 Clearly, the restriction of $\varphi^i, \, i=1,2$  to $\overline{S_T}$ belongs to  $C^1(\overline{S_T})$.
 Since $u$ is a subsolution to \eqref{E:HJ_Equation} and $u-(\varphi^1+\eta t)=
 (u_{\eta}-\varphi^1)(x,t)= \Psi_{\alpha}(x,t,\bar y, \bar s)$ takes a maximum at $(x,t)=(\bar x,\bar t)\in\overline{S_T}$ with $\bar t>0$, we have
  \begin{align*}
    \partial^\bullet\varphi^1(\bar x, \bar t)+H(\bar x, \bar t,\nabla_\Gamma\varphi^1(\bar x, \bar t)) \leq -\eta.
  \end{align*}
Observing that by \eqref{E:Tangential_Gradient} and \eqref{E:Material_Euler}
  \begin{displaymath}
    \nabla_\Gamma\varphi^1(x,t) = \frac{2}{\alpha^2} P_\Gamma(x,t)(x- \bar y), \quad
    \partial^\bullet\varphi^1(x,t) = \frac{2}{\alpha^2}(t- \bar s)+\frac{2}{\alpha^2}v_\Gamma(x,t)\cdot(x-\bar y)
  \end{displaymath}
we deduce    
  \begin{equation} \label{Pf_CP:Subsol_Test}
  \frac{2(\bar t -\bar s)}{\alpha^2}+\frac{2}{\alpha^2}v_\Gamma(\bar x,\bar t)\cdot(\bar x-\bar y) 
    +H\bigl(\bar x,\bar t,\frac{2}{\alpha^2} P_\Gamma(\bar x,\bar t)(\bar x- \bar y)\bigr) \leq -\eta.
  \end{equation}
  Since $v$ is a supersolution and $(v-\varphi^2)(y,s)=-\Psi_{\alpha}(\bar x,\bar t,y,s)$ takes a minimum at $(y,s)=(\bar y,\bar s)\in\overline{S_T}$ with $\bar s>0$, it follows that
  \begin{align*}
    \partial^\bullet\varphi^2(\bar y,\bar s)+H(\bar y,\bar s,\nabla_\Gamma\varphi^2(\bar y,\bar s)) \geq 0
  \end{align*}
  and we obtain similarly as above
  \begin{equation} \label{Pf_CP:Supersol_Test}
   -\frac{2(\bar t- \bar s)}{\alpha^2} -\frac{2}{\alpha^2} v_\Gamma(\bar y, \bar s)\cdot(\bar x- \bar y) 
    -H\bigl(\bar y,\bar s,\frac{2}{\alpha^2} P_\Gamma(\bar y, \bar s)(\bar x- \bar y)\bigr) \leq 0.
  \end{equation}
  We deduce from  \eqref{Pf_CP:Subsol_Test} and \eqref{Pf_CP:Supersol_Test} that 
  \begin{eqnarray}
      \bar A &:= & \frac{2}{\alpha^2}\{v_\Gamma(\bar x, \bar t)-v_\Gamma(\bar y,\bar s)\}\cdot(\bar x-\bar y)
   \nonumber  \\
     & & +  H\left(\bar x,\bar t,\frac{2}{\alpha^2} P_\Gamma(\bar x, \bar t)(\bar x-\bar y)\right)
      -H\left(\bar y,\bar s,\frac{2}{\alpha^2} P_\Gamma(\bar y,\bar s)(\bar x-\bar y)\right) \leq - \eta. 
       \label{Pf_CP:Difference_Test} 
  \end{eqnarray}
 Since $v_{\Gamma}, P_{\Gamma}$ are  smooth on $\overline{S_T}$ we obtain with the help of 
  \eqref{E:Hamiltonian_Lip_xt}, \eqref{E:Hamiltonian_Lip_p} and (\ref{conv1})
\begin{eqnarray*}
    | \bar A| &  \leq & \frac{2}{\alpha^2} \bigl( | v_{\Gamma}(\bar x,\bar t)-v_{\Gamma}(\bar y,\bar s)|
  +L_{H,2} |P_\Gamma(\bar x ,\bar t)-P_\Gamma(\bar y,\bar s)| \bigr)  | \bar x -\bar y| \\  
    & &  \quad +   L_{H,1}(| \bar x - \bar y|+| \bar t - \bar s |)\left(1+\frac{2}{\alpha^2}|P_\Gamma(\bar x,\bar t)(\bar x -\bar y)|\right) \\
    & \leq &   C \frac{| \bar x - \bar y |^2+| \bar t - \bar s |^2}{\alpha^2} + \alpha^2 \rightarrow 0, \alpha \rightarrow 0
  \end{eqnarray*}
contradicting (\ref{Pf_CP:Difference_Test}). Hence, $\sigma_{\eta} \leq 0$, so that $u_{\eta} \leq v$
on $\overline{S_T}$. The result now follows upon sending $\eta \rightarrow 0$.
\end{proof}

\begin{corollary}[Uniqueness of a viscosity solution] \label{C:Uniqueness_Solution}
  For any initial value $u_0\in C(\Gamma(0))$ there exists at most one viscosity solution to \eqref{E:HJ_Equation}.
\end{corollary}

%%%%%%%%%%%
  %%%%%%%%%%%
\begin{comment}
\begin{proof}
  Suppose that $u$ and $v$ are viscosity solutions to \eqref{E:HJ_Equation} with the same initial value $u_0$.
  Then since $u$ and $v$ are a subsolution and supersolution, respectively, and $u^\ast(\cdot,0)=v_\ast(\cdot,0)=u_0$ on 
  $\Gamma(0)$ by Lemmas~\ref{L:Property_Envelope} and~\ref{L:Visc_Sol_Conti}, it follows 
  from Lemma~\ref{L:Comparison_Principle} that $u\leq v$ on $\overline{S_T}$.
  We exchange $u$ with $v$ and argue similarly to get $v\leq u$ on $\overline{S_T}$.
  Hence $u=v$.
\end{proof}
\end{comment}
%%%%%%%%%%%
  %%%%%%%%%%%

%\section{Evolving finite element spaces} \label{S:Finite_Element}

%\subsection{Triangulation of evolving surfaces} \label{SS:Triangulation}

%%%%%%%%%%%%%%%%%%% SECTION %%%%%%%%%%%%%%%%%%

\section{Finite volume scheme} \label{S:Finite_Volume}
\noindent
Let us next turn to the approximation of  \eqref{E:HJ_Equation}. As mentioned already in the introduction,
our scheme is based on the finite volume scheme for Hamilton--Jacobi equations in a flat and stationary domain 
introduced by Kim and Li in \cite{KL15JCM}. \\
Let $t^0<t^1<\ldots <t^{N-1}<t^N=T$ be a partitioning of $[0,T]$ with time steps $\tau^n=t^{n+1}-t^n$
and $\tau:= \max_{n=0,...,N-1} \tau^n$ as well as
$x_i^n=x_i(t^n), \, V_h^n=V_h(t^n)$.  In order
to derive our scheme we start from the following viscous approximation of \eqref{E:HJ_Equation}
\begin{equation}  \label{viscapprox}
  \partial^\bullet u(x,t)+H(x,t,\nabla_\Gamma u(x,t)) = \varepsilon\Delta_\Gamma u(x,t),  \quad (x,t)\in S_T,
\end{equation}
where $0< \varepsilon \ll 1$. Let us fix $i \in \lbrace 1,\ldots,M \rbrace$ and consider a time--dependent
set $V_i(t) \subset \Gamma(t)$ centered at $x_i(t)$. Integrating (\ref{viscapprox}) for $t=t^n$ over
$V_i(t^n)$ we find that
\begin{equation}  \label{deriv1}
\int_{V_i(t^n)}  \partial^\bullet u \, d \mathcal H^2 + \int_{V_i(t^n)} H(\cdot,t^n,\nabla_{\Gamma}u)
\, d \mathcal H^2 = \varepsilon \int_{V_i(t^n)} \Delta_{\Gamma} u \,  d \mathcal H^2.
\end{equation}
Here, $\mathcal H^n$ is the $n$--dimensional Hausdorff measure. 
Let us consider the first term on the left hand side of (\ref{deriv1}). Using the transport theorem (see e.g.
\cite[Theorem 5.1]{DE13Acta}) and
approximating $\int_{V_i(t)} u \, d \mathcal H^2$ by $u(x_i(t),t) | V_i(t) |$  ($| V_i(t)| = \mathcal H^2(V_i(t))$)
we obtain
\begin{eqnarray*}
 \int_{V_i(t^n)}  \partial^\bullet u \, d \mathcal H^2  &= & \frac{d}{dt} \int_{V_i(t)}
u \, d \mathcal H^2_{|t=t^n} - \int_{V_i(t^n)} \nabla_{\Gamma} \cdot v_{\Gamma} \, u \, d \mathcal H^2  \\
& \approx & \frac{  u(x_i^{n+1},t^{n+1}) | V_i(t^{n+1}) | - u(x_i^n,t^n) | V_i(t^n) | }{\tau^n}
-  \int_{V_i(t^n)} \nabla_{\Gamma} \cdot v_{\Gamma} \, u  \, d \mathcal H^2.
\end{eqnarray*}
Since $\displaystyle \frac{d}{dt} | V_i(t) | = \int_{V_i(t)} \nabla_{\Gamma} \cdot v_{\Gamma} \, d \mathcal H^2$ we
may approximate $| V_i(t^{n+1}) | \approx | V_i(t^n) | + \tau^n \int_{V_i(t^n)} \nabla_{\Gamma} \cdot
v_{\Gamma} \, d \mathcal H^2$ so that
\begin{displaymath}
\int_{V_i(t^n)}  \partial^\bullet u  \, d \mathcal H^2 \approx 
 \frac{  u(x_i^{n+1},t^{n+1})  - u(x_i^n,t^n) }{\tau^n} \,  | V_i(t^n) |.
\end{displaymath}
Finally, after applying Gauss theorem for hypersurfaces  to the integral on the right hand side of (\ref{deriv1}) we obtain
\begin{equation}  \label{deriv2}
\frac{  u(x_i^{n+1},t^{n+1})  - u(x_i^n,t^n) }{\tau^n} \,  | V_i(t^n) | +  \int_{V_i(t^n)} H(\cdot,t^n,\nabla_{\Gamma}u)
\, d \mathcal H^2  \approx \varepsilon \int_{\partial V_i(t^n)} \frac{\partial u}{\partial \mu} \, d \mathcal H^1,
\end{equation}
where $\mu$ denotes the outer unit conormal to $\partial V_i(t^n)$. In order to turn (\ref{deriv2}) into
a numerical scheme we construct a suitable discrete version $V^{n,i} \subset \Gamma_h(t^n)$ of
$V_i(t^n)$
and take for $\varepsilon$  a vertex dependent parameter $\varepsilon^n_i$.
%Here $x_i^n$ denotes the point of the node $i$ at time $t^n$.
%We abuse the notation of the linear interpolant $I_h^n:=I_h^{t^n}$, i.e. for a function $u$ on $\Gamma(t^n)$ we %define $I_h^nu\in V_h^n$ as
%\begin{align*}
%  I_h^nu(x) := \sum_{i=1}^Mu(x_i^n)\chi_i(x,t^n), \quad x\in\Gamma_h(t^n).
%\end{align*}
%When $u$ is defined on $\overline{S_T}$, we abbreviate $I_h^nu(\cdot,t^n)$ to $I_h^nu$.
Let $\Upsilon_i\in\mathbb{N}$ be the number of triangles that have the common vertex $i$, which is independent of $n$.
The other vertices of the triangles with common vertex $i$ are denoted by $i_j$, $j=1,\dots,\Upsilon_i$, which we enumerate 
in clockwise direction. We write $T_j^{n,i}  \in \mathcal T_h(t^n)$ for the triangle with vertices $i$, $i_j$, and $i_{j+1}$ 
and  $E_j^{n,i}$
% \bar{E}_j^{n,i}$ 
for the edge of $T_j^{n,i}$ connecting the vertices $i$ and $i_j$
% and the vertices $i_j$ and $i_{j+1}$, respectively 
(see Figure~\ref{F:Vertex_i}, left).
\begin{figure}[htb]
  \centering
   \includegraphics[scale=0.4,clip]{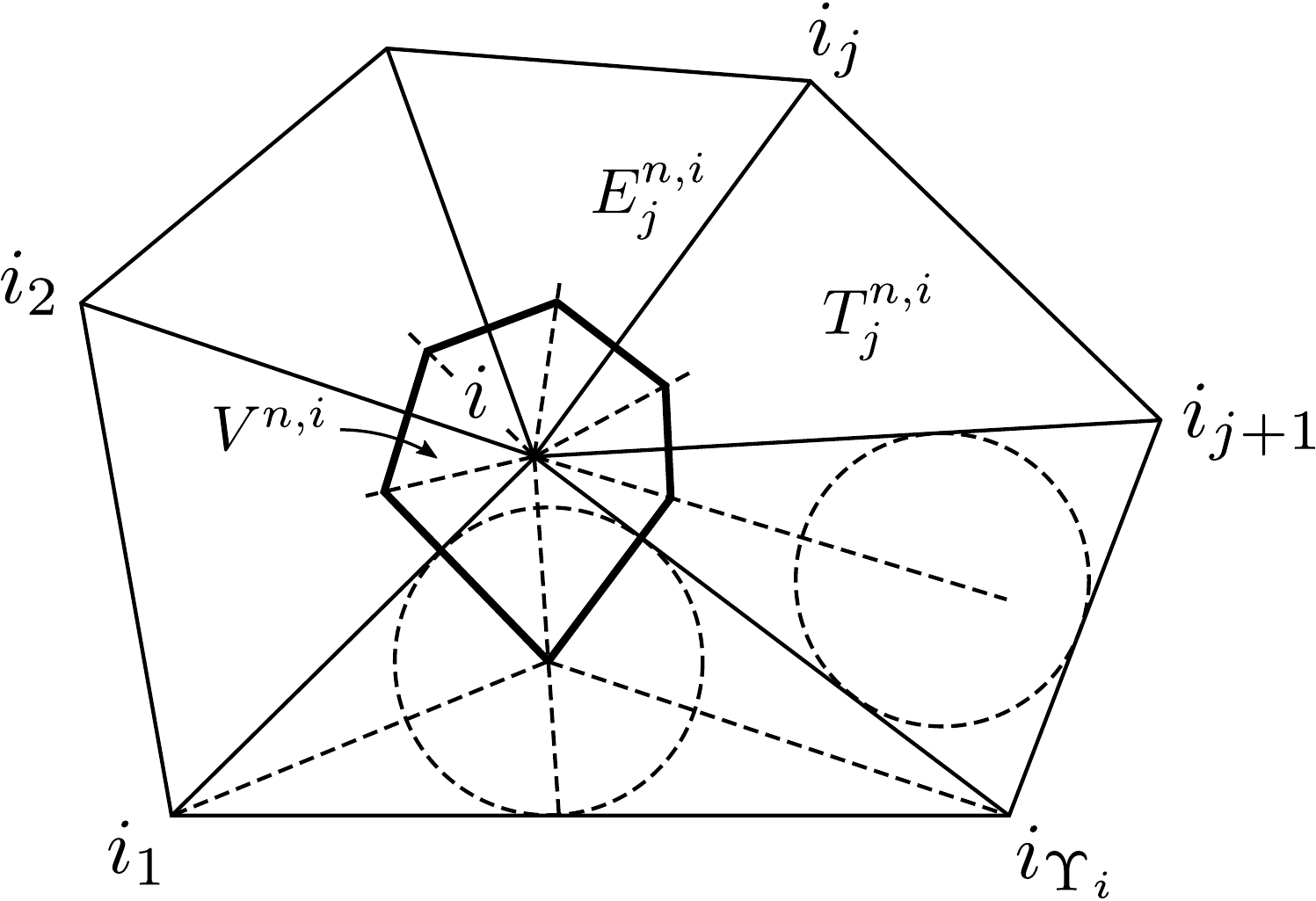}
 % \caption{The vertices, triangles, edges, and volume}
 \hspace{2mm} \hfill
  \includegraphics[scale=0.4,clip]{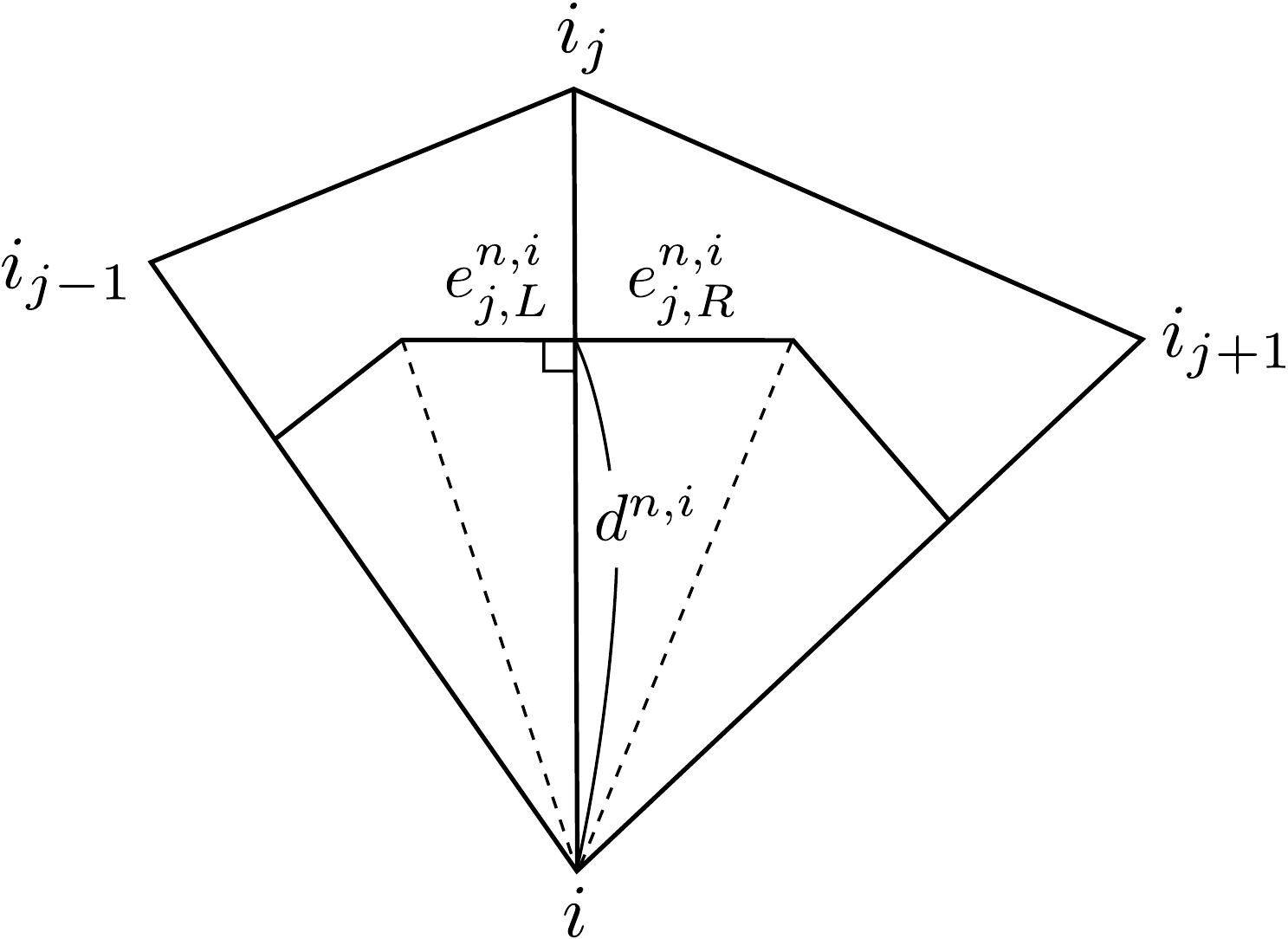}
  \caption{}
    \label{F:Vertex_i}
\end{figure}

\noindent
Let $d_j^{n,i}$ be the length from the vertex $i$ to the contact point on $E_j^{n,i}$ of the inscribed circle of $T_j^{n,i}$ and $d^{n,i}:=\min\{d_j^{n,i} \mid j=1,\dots,\Upsilon_i\}$.
We define the volume $V^{n,i} \subset \Gamma_h(t^n)$ as a polygonal region surrounded by line segments perpendicular to each edge $E_j^{n,i}$ and 
whose distances from the vertex $i$ are all equal to $d^{n,i}$.
The parts of the edge of $V^{n,i}$ perpendicular to $E_j^{n,i}$ and lying in $T_{j-1}^{n,i}$ and $T_j^{n,i}$ are denoted by 
$e_{j,L}^{n,i}$ and $e_{j,R}^{n,i}$ with their length $h_{j,L}^{n,i}$ and $h_{j,R}^{n,i}$, respectively (see Figure~\ref{F:Vertex_i}, right). 
The diameter of $T^{n,i}_j$ is denoted by $h_{T^{n,i}_j}$. Note that in view of (\ref{E:Quotient_Diam_Local}) there exist constants
$0< \alpha_1 < \alpha_2$ and $C>0$  such that
\begin{equation}  \label{E:Quotient_EV}
    \alpha_1 \leq \frac{h_{j,L}^{n,i}+h_{j,R}^{n,i}}{|E_j^{n,i}|} \leq \alpha_2, \quad h_{T^{n,i}_j} \leq C d^{n,i} 
%    \alpha_1  \leq \frac{h^2}{|V^{n,i}|} \leq \alpha_2, 
\end{equation}
for all $n=0,1,\dots,N$, $i=1,\dots,M$, and $j=1,\dots,\Upsilon_i$.

\noindent
If we look for a discrete solution of the form
 $u^n_h = \sum_{i=1}^M u^n_i \chi_i(\cdot,t^n) \in V^n_h$, then (\ref{deriv2}) motivates the following
 relation:
\begin{displaymath}
  \frac{u_i^{n+1}-u_i^n}{\tau^n}|V^{n,i}|+\sum_{j=1}^{\Upsilon_i}|V^{n,i}\cap T_j^{n,i}| \, H\left(x_i^n,t^n,\nabla_{\Gamma_h}u_h^n|_{T_j^{n,i}}\right) 
  = \varepsilon^n_i\sum_{j=1}^{\Upsilon_i}\frac{u_{i_j}^n-u_i^n}{|E_j^{n,i}|}(h_{j,L}^{n,i}+h_{j,R}^{n,i}),
\end{displaymath}
where we allow the coefficient $\varepsilon^n_i>0$ to depend on the time step and the vertex. Let
$\nu^{n,i}_j= \nu_{h| T^{n,i}_j}$ and
 $\nabla_{\Gamma_h} u^n=(I_3-  \nu_j^{n,i} \otimes  \nu_j^{n,i}) \nabla u^n_h$.
Note that $\nu^{n,i}_j$ and hence $ \nabla_{\Gamma_h} u^n$
is constant on $T_j^{n,i}$. To summarize, our numerical scheme for the Hamilton--Jacobi equation \eqref{E:HJ_Equation} looks as follows.
For a given $u_0\colon\Gamma(0)\to\mathbb{R}$, set
\begin{align} \label{E:Scheme_Initial}
  u_h^0 := I_h^0u_0 = \sum_{i=1}^Mu_i^0\chi_i(\cdot,0) \in V_h^0, \quad u_i^0 := u_0(x_i^0).
\end{align}
For $n=0,1,\dots,N-1$, if $u_h^n=\sum_{i=1}^Mu_i^n\chi_i(\cdot,t^n)\in V_h^n$ is given, then we define
\begin{align} \label{E:Scheme_Next}
  u_h^{n+1} = S_h^n(u_h^n) := \sum_{i=1}^Mu_i^{n+1}\chi_i(\cdot,t^{n+1}) \in V_h^{n+1}
\end{align}
where
\begin{align} \label{E:Scheme_Iteration}
  u_i^{n+1} = [S_h^n(u_h^n)]_i := u_i^n-\tau^n H_i^n(u_i^n,u_{i_1}^n,\dots,u_{i_{\Upsilon_i}}^n), \quad 
  i=1,\dots,M.
\end{align}
Here $H_i^n(u_i^n,u_{i_1}^n,\dots,u_{i_{\Upsilon_i}}^n)$ is the numerical Hamiltonian given by
\begin{equation} \label{E:Numerical_Hamiltonian}
  H_i^n(u_i^n,u_{i_1}^n,\dots,u_{i_{\Upsilon_i}}^n) := \sum_{j=1}^{\Upsilon_i}\frac{|V^{n,i}\cap T_j^{n,i}|}{|V^{n,i}|}H\left(x_i^n,t^n,\nabla_{\Gamma_h}u_h^n|_{T_j^{n,i}}\right)
  -\frac{\varepsilon^n_i}{|V^{n,i}|}\sum_{j=1}^{\Upsilon_i}\frac{u_{i_j}^n-u_i^n}{|E_j^{n,i}|}(h_{j,L}^{n,i}+h_{j,R}^{n,i}).
\end{equation}

%\subsection{Monotonicity, consistency, and invariance under translation with constants} %\label{SS:Property_Scheme}

%%%%%%%%%%%%%%
%%%%%%%%%%%%%%
\begin{comment}
First we observe that for each $i=1,\dots,M$ and $j=1,\dots,\Upsilon_i$,
\begin{align*}
  u_h^n|_{T_j^{n,i}} = (u_{i_j}^n-u_i^n)\chi_{i_j}(\cdot,t^n)|_{T_j^{n,i}}+(u_{i_{j+1}}^n-u_i^n)\chi_{i_{j+1}}(\cdot,t^n)|_{T_j^{n,i}}
\end{align*}
by $\chi_k(\cdot,t^n)=0$ on $T_j^{n,i}$ for $k\neq i$, $i_j$, $i_{j+1}$ and
\begin{align*}
  \chi_i(\cdot,t^n)+\chi_{i_j}(\cdot,t^n)+\chi_{i_{j+1}}(\cdot,t^n) = 1 \quad\text{on}\quad T_j^{n,i}.
\end{align*}
Hereafter we suppress $t^n$.
By the above formula we have
\begin{align} \label{E:Grad_Uhn}
  \nabla_{\Gamma_h}u_h^n|_{T_j^{n,i}} = (u_{i_j}^n-u_i^n)\nabla_{\Gamma_h}\chi_{i_j}|_{T_j^{n,i}}+(u_{i_{j+1}}^n-u_i^n)\nabla_{\Gamma_h}\chi_{i_{j+1}}|_{T_j^{n,i}}.
\end{align}
Note that $\nabla_{\Gamma_h}\chi_{i_j}|_{T_j^{n,i}}$ and $\nabla_{\Gamma_h}\chi_{i_{j+1}}|_{T_j^{n,i}}$ are constant vectors (see Lemma~\ref{L:Tan_Grad_Nodal}).
\end{comment}
%%%%%%%%%
%%%%%%%%%

\noindent
Let us  derive several properties of the finite volume scheme \eqref{E:Scheme_Initial}--\eqref{E:Numerical_Hamiltonian}. It is easy to see that the
scheme is invariant under translation with constants, i.e.
\begin{align} \label{E:Invariance_Constant}
  S_h^n(u_h^n+c)=S_h^n(u_h^n)+c
\end{align}
for any $u_h^n\in V_h^n$ and $c\in\mathbb{R}$.  We proceed by proving that the scheme is monotone.

\begin{lemma}[Monotonicity] \label{L:Monotonicity}
%  Assume that \ref{A:Triangle_Diameter} are \ref{A:Minimal_Angle} are satisfied.
There exist positive constants $C_0,C_1$ and $C_2$ depending only on $\gamma$ and $L_{H,2}$ such that, if
  \begin{align} \label{E:Epsilon_Tau}
  C_0 \max_j h_{T^{n,i}_j} \leq   \varepsilon^n_i  \leq  C_1 \max_j h_{T^{n,i}_j}, \quad \tau^n \leq C_2 \min_{i,j} | E^{n,i}_j |
  \end{align}
  and $u_h^n$, $v_h^n\in V_h^n$ satisfy $u_h^n\leq v_h^n$ on $\Gamma_h(t^n)$, then $S_h^n(u_h^n)\leq S_h^n(v_h^n)$ on $\Gamma_h(t^{n+1})$.
\end{lemma}

\begin{proof}
  Let $u_h^n$, $v_h^n\in V_h^n$ be of the form
  \begin{align*}
    u_h^n = \sum_{i=1}^Mu_i^n\chi_i(\cdot,t^n), \quad v_h^n = \sum_{i=1}^Mv_i^n\chi_i(\cdot,t^n) \quad\text{on}\quad \Gamma_h(t^n).
  \end{align*}
Note  that $u_h^n\leq v_h^n$ on $\Gamma_h(t^n)$
  is equivalent to $u_i^n\leq v_i^n$ for all $i=1,\dots,M$ since the nodal basis functions $\chi_i$ are piecewise linear affine and satisfy $\chi_i(x_j(t),t)=\delta_{ij}$.
  By the same reason it is sufficient to establish that
  \begin{align} \label{Pf_M:Ineq_Component}
    [S_h^n(u_h^n)]_i \leq [S_h^n(v_h^n)]_i \quad\text{for all}\quad i=1,\dots,M
  \end{align}
 in order to prove our claim.
  For $i=1,\dots,M$, by \eqref{E:Scheme_Iteration} and \eqref{E:Numerical_Hamiltonian} we have
  \begin{align} \label{Pf_M:Diff_Component}
    [S_h^n(v_h^n)]_i-[S_h^n(u_h^n)]_i = v_i^n-u_i^n+\tau^n(I_1+I_2+I_3),
  \end{align}
  where $I_1+I_2+I_3=-H_i^n(v_i^n,v_{i_1}^n,\dots,v_{i_{\Upsilon_i}}^n)+H_i^n(u_i^n,u_{i_1}^n,\dots,u_{i_{\Upsilon_i}}^n)$ with
  \begin{align*}
    I_1 &:= -\sum_{j=1}^{\Upsilon_i}\frac{|V^{n,i}\cap T_j^{n,i}|}{|V^{n,i}|}\left\{H\left(\nabla_{\Gamma_h}v_h^n|_{T_j^{n,i}}\right)-H\left(\nabla_{\Gamma_h}u_h^n|_{T_j^{n,i}}\right)\right\}, \\
    I_2 &:= \frac{\varepsilon^n_i}{|V^{n,i}|}\sum_{j=1}^{\Upsilon_i}\frac{v_{i_j}^n-u_{i_j}^n}{|E_j^{n,i}|}(h_{j,L}^{n,i}+h_{j,R}^{n,i}), \\
    I_3 &:= -\frac{\varepsilon^n_i(v_i^n-u_i^n)}{|V^{n,i}|}\sum_{j=1}^{\Upsilon_i}\frac{(h_{j,L}^{n,i}+h_{j,R}^{n,i})}{|E_j^{n,i}|}.
  \end{align*}
  In the definition of $I_1$ we suppressed $x_i^n$ and $t^n$ of $H$.
  Let us estimate $I_1$, $I_2$, and $I_3$.
  By \eqref{E:Hamiltonian_Lip_p} and an inverse inequality
\begin{eqnarray*}
  \lefteqn{ \hspace{-3.4cm}
    \left|H\left(\nabla_{\Gamma_h}v_h^n|_{T_j^{n,i}}\right)-H\left(\nabla_{\Gamma_h}u_h^n|_{T_j^{n,i}}\right)\right| \leq L_{H,2}\left|\nabla_{\Gamma_h}v_h^n|_{T_j^{n,i}}-\nabla_{\Gamma_h}u_h^n|_{T_j^{n,i}}\right| 
 %   &  \leq & C | \nabla (v^n_h -u^n_h)_{| T^{n,i}_j} |
     \leq C | E^{n,i}_j|^{-1} \, \Vert   v^n_h - u^n_h    \Vert_{B(T^{n,i}_j)} } \\
   &\leq &  C | E^{n,i}_j|^{-1}  \{ (v_i^n-u_i^n)+(v_{i_j}^n-u_{i_j}^n)+(v_{i_{j+1}}^n-u_{j_{i+1}}^n)\},  
\end{eqnarray*}
since $u^n_h,v^n_h$ are linear on $T^{n,i}_j$ and $v^n_h -u^n_h \geq 0$.
\begin{comment}
  To the right-hand side we apply the formula \eqref{E:Grad_Uhn} and the inequalities $u_k^n\leq v_k^n$ for $k=1,\dots,M$ and \eqref{E:Grad_Nodal_Norm}.
  Then
  \begin{multline*}
    \left|H\left(\nabla_{\Gamma_h}v_h^n|_{T_j^{n,i}}\right)-H\left(\nabla_{\Gamma_h}u_h^n|_{T_j^{n,i}}\right)\right| \\
    \begin{aligned}
      &\leq L_{H,2}\left\{(v_i^n-u_i^n)\left(\left|\nabla_{\Gamma_h}\chi_{i_j}|_{T_j^{n,i}}\right|+\left|\nabla_{\Gamma_h}\chi_{i_{j+1}}|_{T_j^{n,i}}\right|\right)\right. \\
      &\qquad\qquad \left.+(v_{i_j}^n-u_{i_j}^n)\left|\nabla_{\Gamma_h}\chi_{i_j}|_{T_j^{n,i}}\right|+(v_{i_{j+1}}^n-u_{j_{i+1}}^n)\left|\nabla_{\Gamma_h}\chi_{i_{j+1}}|_{T_j^{n,i}}\right|\right\} \\
      &\leq \frac{L_{H,2}}{h\sin^2\theta_0}\{2(v_i^n-u_i^n)+(v_{i_j}^n-u_{i_j}^n)+(v_{i_{j+1}}^n-u_{j_{i+1}}^n)\}.
    \end{aligned}
  \end{multline*}
\end{comment}
%%%%%%%

\noindent
Using that $\sum_{j=1}^{\Upsilon_i}\frac{|V^{n,i}\cap T_j^{n,i}|}{|V^{n,i}|} = 1$ as well as
  \begin{align} \label{Pf_M:Area_VcapT}
 |V^{n,i}\cap T_j^{n,i} | =\frac{1}{2} d^{n,i} ( h^{n,i}_{j,R} + h^{n,i}_{j+1,L}) \leq 
     | E^{n,i}_j | \, \max_j h_{T^{n,i}_j}, \quad j=1,\ldots,\Upsilon_i,
  \end{align}
  we get
  \begin{equation} \label{Pf_M:Estimate_I1}
      |I_1| \leq \frac{C}{\min_{j} | E^{n,i}_j |} (v_i^n-u_i^n)+ \frac{C}{| V^{n,i} |} \, \max_j h_{T^{n,i}_j}  
   \sum_{j=1}^{\Upsilon_i}(v_{i_j}^n-u_{i_j}^n).
  \end{equation}
Next, from (\ref{E:Quotient_EV}) and the fact that   $u_{i_j}^n\leq v_{i_j}^n$ for $j=1,\dots,\Upsilon_i$  we infer that
  \begin{equation} \label{Pf_M:Estimate_I2}
    I_2 \geq \frac{\alpha_1 \varepsilon^n_i}{| V^{n,i} |} \, \sum_{j=1}^{\Upsilon_i}(v_{i_j}^n-u_{i_j}^n).
  \end{equation}
%  Applying the right-hand inequality of \eqref{E:Node_Number} to $1/\Upsilon_i$ we further get
 % \begin{align} \label{Pf_M:Estimate_I2}
 %   I_2 \geq \frac{\varepsilon}{h^2}\cdot\frac{\alpha_1\alpha_3\theta_0}{2\pi}\sum_{j=1}^{\Upsilon_i}(v_{i_j}^n-u_{i_j}^n).
%  \end{align}
%  Also, by $u_i^n\leq v_i^n$ and the right-hand inequalities of \eqref{E:Quotient_Edges} and %\eqref{E:Quotient_Volume} we have
In view of the relation   $| V^{n,i} | = \sum_{j=1}^{\Upsilon_i} \frac{1}{2} d^{n,i} ( h^{n,i}_{j,L} + h^{n,i}_{j,R})$
and (\ref{E:Quotient_Diam_Local})  we obtain
\begin{eqnarray*}
\frac{1}{| V^{n,i} |} \sum_{j=1}^{\Upsilon_i} \frac{h_{j,L}^{n,i}+h_{j,R}^{n,i}}{|E_j^{n,i}|} & \leq &
\frac{1}{| V^{n,i} |}  \frac{1}{\min_j | E^{n,i}_j |} \sum_{j=1}^{\Upsilon_i} (h_{j,L}^{n,i}+h_{j,R}^{n,i}) 
 =  \frac{2}{d^{n,i}}   \frac{1}{\min_j | E^{n,i}_j |} \\
 &  \leq &   \frac{C}{\max_j h_{T^{n,i}_j} \min_j | E^{n,i}_j |}
\leq \frac{C C_1}{\varepsilon^n_i} \frac{1}{\min_j | E^{n,i}_j |},
\end{eqnarray*}
where we used  (\ref{E:Epsilon_Tau}) in the last step. Hence
  \begin{align} \label{Pf_M:Estimate_I3}
    I_3 \geq -  \frac{C C_1}{\min_j | E^{n,i}_j |}(v_i^n-u_i^n).
  \end{align}
   From \eqref{Pf_M:Diff_Component}, \eqref{Pf_M:Estimate_I1}, \eqref{Pf_M:Estimate_I2}, and \eqref{Pf_M:Estimate_I3} it follows that
  \begin{displaymath}  
    [S_h^n(v_h^n)]_i-[S_h^n(u_h^n)]_i 
     \geq   \bigl( 1-\frac{\tau^n C(1+C_1)}{ \min_j | E^{n,i}_j | } \bigr)(v_i^n-u_i^n) 
    + \frac{1}{ | V^{n,i} |} (\alpha_1 \varepsilon^n_i - C \max_j h_{T^{n,i}_j}) \, \sum_{j=1}^{\Upsilon_i}(v_{i_j}^n-u_{i_j}^n) 
  \end{displaymath}
  which yields (\ref{Pf_M:Ineq_Component})
  if we choose $C_0= \frac{C}{\alpha_1}$ and $C_2= \frac{1}{C(1+C_1)}$ in (\ref{E:Epsilon_Tau}).
\end{proof}

\begin{comment}
\begin{remark} \label{R:Const_ET}
  By $\gamma=3/\sin^3\theta_0$ and \eqref{E:Quotient_Constants} the constants $C_1$ and $C_2$ given in \eqref{Pf_M:Const_ET} are expressed in terms of $c_K$, $\theta_0$, and $L_H$ as
  \begin{align} \label{E:Const_ET_Explicit}
    C_1 = \frac{6\pi c_KL_{H,2}}{\theta_0c_1(\theta_0)}, \quad C_2 = \frac{\theta_0c_2(\theta_0)\sin^2\theta_0}{2L_{H,2}\{54\pi c_K^3+\theta_0c_2(\theta_0)\}}.
  \end{align}
  where $c_1(\theta_0):=\sin^5\theta_0\tan^2\theta_0\tan^3\theta_0/2$ and $c_2(\theta_0):=c_1(\theta_0)^2\tan\theta_0$.
  Hence we easily see that $C_1\to\infty$ and $C_2\to0$ as $\theta_0\to0$, i.e. to get the monotonicity we should take $\varepsilon/h$ larger and $\tau/h$ smaller as $\theta_0$ becomes smaller.
\end{remark}
\end{comment}

\noindent
In what follows we write $I^n_h \varphi$ instead of $I^{t^n}_h \varphi$, i.e.
\begin{align*}
    I_h^n\varphi = \sum_{i=1}^M\varphi_i^n\chi_i(\cdot,t^n) \in V_h^n, \quad \varphi_i^n = \varphi(x_i^n,t^n).
  \end{align*}

\begin{lemma}[Consistency] \label{L:Consistency}
 Suppose that  \eqref{E:Epsilon_Tau} is satisfied.
  Then there exists a constant $C_3>0$ depending only on $\gamma$, $L_{H,2}$ such that
  \begin{multline} \label{E:Consistency}
    \left|\frac{\varphi_i^{n+1}-[S_h^n(I_h^n\varphi)]_i}{\tau^n}-\{  \partial^\bullet \varphi(x_i^n,t^n)+H(x_i^n,t^n,\nabla_\Gamma\varphi(x_i^n,t^n))\}\right| \\
    \leq C_3h\left(\|\nabla_\Gamma\varphi\|_{B(\overline{S_T})}+\|\nabla_\Gamma^2\varphi\|_{B(\overline{S_T})}+\|(\partial^\bullet)^2 \varphi\|_{B(\overline{S_T})}\right)
  \end{multline}
  for all $\varphi\in C^2(\overline{S_T})$, $n=0,1,\dots,N-1$, and $i=1,\dots,M$.
  Here,  $(\partial^\bullet)^2 \varphi$ is the second-order material derivative of $\varphi$.
\end{lemma}

\begin{proof}
Using    \eqref{E:Scheme_Iteration} and \eqref{E:Numerical_Hamiltonian} we have
  \begin{align*}
    \frac{\varphi_i^{n+1}-[S_h^n(I_h^n\varphi)]_i}{\tau^n} = 
    \frac{\varphi_i^{n+1}-\varphi_i^n}{\tau^n}+H_i^n(\varphi_i^n,\varphi_{i_1}^n,\dots,\varphi_{i_{\Upsilon_i}}^n).
  \end{align*}
  Let us set
  \begin{align*}
    I_1 &:= \frac{\varphi_i^{n+1}-\varphi_i^n}{\tau^n}-\partial^\bullet \varphi(x_i^n,t^n), \\
    I_2 &:= \sum_{j=1}^{\Upsilon_i}\frac{|V^{n,i}\cap T_j^{n,i}|}{|V^{n,i}|}H(x_i^n,t^n,\nabla_{\Gamma_h}I_h^n\varphi|_{T_j^{n,i}})-H(x_i^n,t^n,\nabla_\Gamma\varphi(x_i^n,t^n)), \\
    I_3 &:= -\frac{\varepsilon^n_i}{|V^{n,i}|}\sum_{j=1}^{\Upsilon_i}\frac{\varphi_{i_j}^n-\varphi_i^n}{|E_j^{n,i}|}(h_{j,L}^{n,i}+h_{j,R}^{n,i}),
  \end{align*}
  so that
  \begin{equation}  \label{i1i2i3}
    \frac{\varphi_i^{n+1}-[S_h^n(I_h^n\varphi)]_i}{\tau^n}-\{ \partial^\bullet \varphi(x_i^n,t^n)+H(x_i^n,t^n,\nabla_\Gamma\varphi(x_i^n,t^n))\} = I_1+I_2+I_3
  \end{equation}
  and estimate $I_1$, $I_2$, and $I_3$ separately.
  From $ \varphi_i^n = \varphi(x_i^n,t^n) = \varphi(\Phi(x_i^0,t^n),t^n)$ 
  and the definition of the material derivative it follows that
  \begin{align*}
    \frac{\varphi_i^{n+1}-\varphi_i^n}{\tau^n} =
     \frac{1}{\tau^n}\int_{t^n}^{t^{n+1}}\frac{d}{dt}\bigl(\varphi(\Phi(x_i^0,s),s)\bigr)\,ds = 
     \frac{1}{\tau^n}\int_{t^n}^{t^{n+1}}\partial^\bullet \varphi(\Phi(x_i^0,s),s)\,ds.
  \end{align*}
  Applying the definition of the material derivative again we obtain
  \begin{displaymath}
    I_1  
    % \frac{1}{\tau}\int_{t^n}^{t^{n+1}}\dot{\varphi}(\Phi(x_i^0,s),s),ds-\dot{\varphi}(x_i^n,t^n) \\
    =  \frac{1}{\tau^n}\int_{t^n}^{t^{n+1}}\{\partial^\bullet \varphi(\Phi(x_i^0,s),s)-\partial^\bullet \varphi(\Phi(x_0^i,t^n),t^n)\}ds  
    =  \frac{1}{\tau^n}\int_{t^n}^{t^{n+1}} 
    \int_{t^n}^s (\partial^\bullet)^2 \varphi(\Phi(x_i^0,\tilde{s}),\tilde{s})\,d\tilde{s} \, ds.
  \end{displaymath}
  Since $\varphi\in C^2(\overline{S_T})$, the second-order material derivative $(\partial^\bullet)^2 \varphi$ is bounded on $\overline{S_T}$.
  Hence by the above equality, $t^{n+1}-t^n=\tau^n$, and \eqref{E:Epsilon_Tau} we obtain
  \begin{align} \label{Pf_C:Estimate_I1}
    |I_1| \leq \frac{(t^{n+1}-t^n)^2}{\tau^n}\| (\partial^\bullet)^2 \varphi\|_{B(\overline{S_T})} 
    = \tau^n \| (\partial^\bullet)^2 \varphi\|_{B(\overline{S_T})} \leq 
    C h\| (\partial^\bullet)^2 \varphi\|_{B(\overline{S_T})}.
  \end{align}
  Next we estimate $I_2$.
  From now on, we suppress $t^n$ in all functions and $x_i^n$ in the Hamiltonian. Clearly,
  %Using \eqref{Pf_M:Area_VcapT} we represent $I_2$ as
  \begin{align} \label{Pf_C:Formula_I2}
    I_2 = \sum_{j=1}^{\Upsilon_i}\frac{|V^{n,i}\cap T_j^{n,i}|}{|V^{n,i}|}\left\{H\left(\nabla_{\Gamma_h}I_h^n\varphi|_{T_j^{n,i}}\right)-H(\nabla_\Gamma\varphi(x_i^n))\right\}.
  \end{align}
  For each $j=1,\dots,\Upsilon_i$, the inequality \eqref{E:Hamiltonian_Lip_p} yields
  \begin{align} \label{Pf_C:Est1_H_Phi}
    \left|H\left(\nabla_{\Gamma_h}I_h^n\varphi|_{T_j^{n,i}}\right)-H(\nabla_\Gamma\varphi(x_i^n))\right| \leq L_{H,2}\left|\nabla_{\Gamma_h}I_h^n\varphi|_{T_j^{n,i}}-\nabla_\Gamma\varphi(x_i^n)\right|.
  \end{align}

\noindent
Abbreviating $\varphi^{-l}(x):= \varphi(\pi_h(x)), x \in \Gamma_h$ we may write
\begin{equation}  \label{i2est}
 \nabla_{\Gamma_h}I_h^n\varphi|_{T_j^{n,i}}-\nabla_\Gamma\varphi(x_i^n) = \bigl(
\nabla_{\Gamma_h}I_h^n\varphi|_{T_j^{n,i}}-\nabla_{\Gamma_h} \varphi^{-l}(x^n_i) \bigr) 
 + \bigl( \nabla_{\Gamma_h} \varphi^{-l}(x^n_i) - \nabla_\Gamma\varphi(x_i^n) \bigr) \equiv
A+B. 
\end{equation}
Since $I^n_h \varphi_{| T^{n,i}_j}$ is the linear interpolation of $\varphi^{-l}_{|T^{n,i}_j}$ we obtain
\begin{equation}  \label{aest}
| A | \leq C h \Vert \nabla_{\Gamma_h}^2 \varphi^{-l} \Vert_{B(T^{n,i}_j)} \leq Ch \bigl( \|\nabla_\Gamma \varphi\|_{B(\overline{S_T})}+\|\nabla_\Gamma^2\varphi\|_{B(\overline{S_T})} \bigr).
\end{equation}
On the other hand, we infer  from (4.18) in \cite{DE13Acta} and the relations 
$\pi_h(x^n_i)=x^n_i, d(x^n_i)=0, \, \nu(x^n_i) \cdot \nabla_{\Gamma} \varphi(x^n_i)=0$ that 
\begin{displaymath}
B =  (I_3 - \nu^{n,i}_j \otimes \nu^{n,i}_j) \nabla_\Gamma\varphi(x_i^n) - \nabla_\Gamma\varphi(x_i^n)  
 =  \bigl( \nu(x^n_i) \otimes \nu(x^n_i) -  \nu^{n,i}_j \otimes \nu^{n,i}_j \bigr) \nabla_\Gamma\varphi(x_i^n),
\end{displaymath}
so that by (\ref{nuest})
\begin{equation}  \label{best}
| B | \leq 2 \Vert  \nu - \nu_h \Vert_{B(T^{n,i}_j)}  \|\nabla_\Gamma \varphi\|_{B(\overline{S_T})} \leq C  h  \|\nabla_\Gamma \varphi\|_{B(\overline{S_T})}.
\end{equation}
\begin{comment}
  From \eqref{Pf_C:Est2_H_Phi}, \eqref{Pf_C:Diff_Phi_A}, and \eqref{Pf_C:Est_Phi_C} it follows that
  \begin{multline} \label{Pf_C:Est3_H_Phi}
    \left|H\left(\nabla_{\Gamma_h}I_h^n\varphi|_{T_j^{n,i}}\right)-H(\nabla_\Gamma\varphi(x_i^n))\right| \\
    \leq hL_{H,2}\left(\frac{2c_d}{\sin^2\theta_0}+c_\nu\right)\left(\|\nabla_\Gamma \varphi\|_{B(\overline{S_T})}+\|\nabla_\Gamma^2\varphi\|_{B(\overline{S_T})}\right).
  \end{multline}
\end{comment}
%  Hence we apply \eqref{Pf_M:Area_VcapT} and \eqref{Pf_C:Est3_H_Phi} to \eqref{Pf_C:Formula_I2} to get
Combining (\ref{Pf_C:Formula_I2})--(\ref{best}) we obtain
  \begin{align} \label{Pf_C:Estimate_I2}
    |I_2| \leq Ch \left(\|\nabla_\Gamma \varphi\|_{B(\overline{S_T})}+\|\nabla_\Gamma^2\varphi\|_{B(\overline{S_T})}\right).
  \end{align}
Finally, let us write
 \begin{align} \label{Pf_C:Formula_I3}
    I_3 = \frac{\varepsilon^n_i}{|V^{n,i}|}(J_1+J_2),
  \end{align}
where
  \begin{align} \label{Pf_C:Def_J1_J2}
    \begin{aligned}
      J_1 &:= -\sum_{j=1}^{\Upsilon_i}\frac{h_{j,L}^{n,i}+h_{j,R}^{n,i}}{|E_j^{n,i}|}\{(\varphi_{i_j}^n-\varphi_i^n)-
      \nabla_{\Gamma} \varphi(x^n_i) \cdot  (x^n_{i_j} - x^n_i) \}, \\
      J_2 &:= -\sum_{j=1}^{\Upsilon_i}\left(\nabla_\Gamma\varphi(x_i^n)\cdot\frac{x_{i_j}^n-x_i^n}{|E_j^{n,i}|}\right)(h_{j,L}^{n,i}+h_{j,R}^{n,i}).
    \end{aligned}
  \end{align}
Extending $\varphi$ constantly in normal direction via $\varphi_c$ and recalling (\ref{E:Ext_f}) we have
  \begin{eqnarray*}
\lefteqn{ \hspace{-7mm}    \varphi_{i_j}^n-\varphi_i^n-  \nabla_{\Gamma} \varphi(x^n_i) \cdot  (x^n_{i_j} - x^n_i)  =
\varphi_c(x^n_{i_j}) - \varphi_c(x^n_i)- \nabla \varphi_c(x^n_i) \cdot  (x^n_{i_j} - x^n_i)  } \\
      &=&  \int_0^1\{\nabla\varphi_c(x_i^n+s(x_{i_j}^n-x_i^n))-\nabla\varphi_c(x_i^n)\}\,ds \cdot (x_{i_j}^n-x_i^n) \\
      &= &\int_0^1\left(\int_0^s\nabla^2\varphi_c(x_i^n+\tilde{s}(x_{i_j}^n-x_i^n))(x_{i_j}^n-x_i^n)\,d\tilde{s}\right)ds
      \cdot (x_{i_j}^n-x_i^n).
\end{eqnarray*}
Thus, we deduce from (\ref{E:Estimate_Hess_Ext}) and (\ref{E:Quotient_EV}) that
  \begin{align} \label{Pf_C:Estimate_J1}
    |J_1| \leq C \left(\|\nabla_\Gamma \varphi\|_{B(\overline{S_T})}+\|\nabla_\Gamma^2\varphi\|_{B(\overline{S_T})}\right) \sum_{j=1}^{\Upsilon_i}  | E^{n,i}_j|^2.
  \end{align}
  To estimate $J_2$ we observe that
  \begin{equation}  \label{divp}
    0 = \int_{V^{n,i}}\mathrm{div}_{\Gamma_h}p\,d\mathcal{H}^2 = \sum_{j=1}^{\Upsilon_i}\int_{V^{n,i}\cap T_j^{n,i}}\mathrm{div}_{\Gamma_h}p\,d\mathcal{H}^2
  \end{equation}
  for the constant vector $p=\nabla_\Gamma\varphi(x_i^n)\in\mathbb{R}^3$.
  For each $j=1,\dots,\Upsilon_i$, $V^{n,i}\cap T_j^{n,i}$ is a flat quadrilateral whose sides consist of the edges $e_{j,R}^{n,i}$, $e_{j+1,L}^{n,i}$, and
  \begin{align*}
    S_{j,L}^{n,i} := E_j^{n,i}\cap\partial(V^{n,i}\cap T_j^{n,i}), \quad S_{j,R}^{n,i} :=E^{n,i}_{j+1}\cap\partial(V^{n,i}\cap T_j^{n,i}).
  \end{align*}
  The unit outward co-normal $\mu_j^{n,i}$ to $\partial(V^{n,i}\cap T_j^{n,i})$ (i.e. the unit outward normal to $\partial(V^{n,i}\cap T_j^{n,i})$ that is tangent to $T_j^{n,i}$) is given by
  \begin{align*}
    \mu_j^{n,i} =
    \begin{cases}
      \mu_{j,E}^{n,i} &\text{on}\quad e_{j,R}^{n,i}, \\
      \mu_{j+1,E}^{n,i} &\text{on}\quad e_{j+1,L}^{n,i}, \\
      \mu_{j,L}^{n,i} &\text{on}\quad S_{j,L}^{n,i}, \\
      \mu_{j+1,R}^{n,i} &\text{on}\quad S_{j,R}^{n,i},
    \end{cases}
  \end{align*}
  where  (see Figure~\ref{F:Pf_C_VcapT})
  \begin{align} \label{E:Co-normal_Edge}
\mu_{j,E}^{n,i} := \frac{x_{i_j}^n-x_i^n}{|x_{i_j}^n-x_i^n|} \quad \mbox{and} \quad   \mu_{j,L}^{n,i} := \nu_j^{n,i}\times\mu_{j,E}^{n,i}, \quad \mu_{j,R}^{n,i} := -\nu_{j-1}^{n,i}\times\mu_{j,E}^{n,i}.
\end{align}

 \begin{figure}[tb]
  \centering
  \includegraphics[scale=0.65,clip]{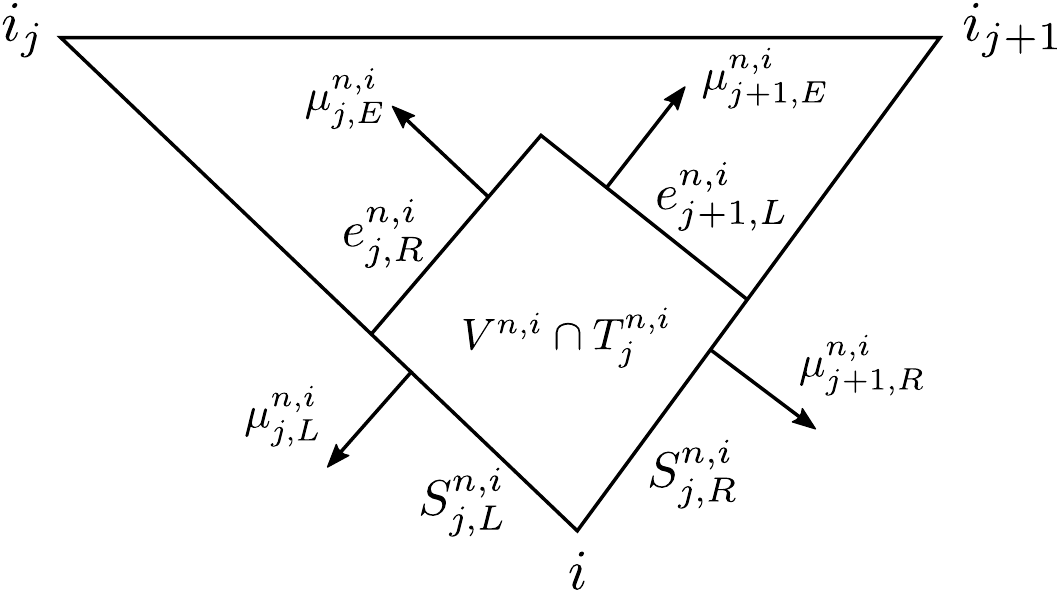}
  \hspace{2mm} \hfill
   \includegraphics[scale=0.65,clip]{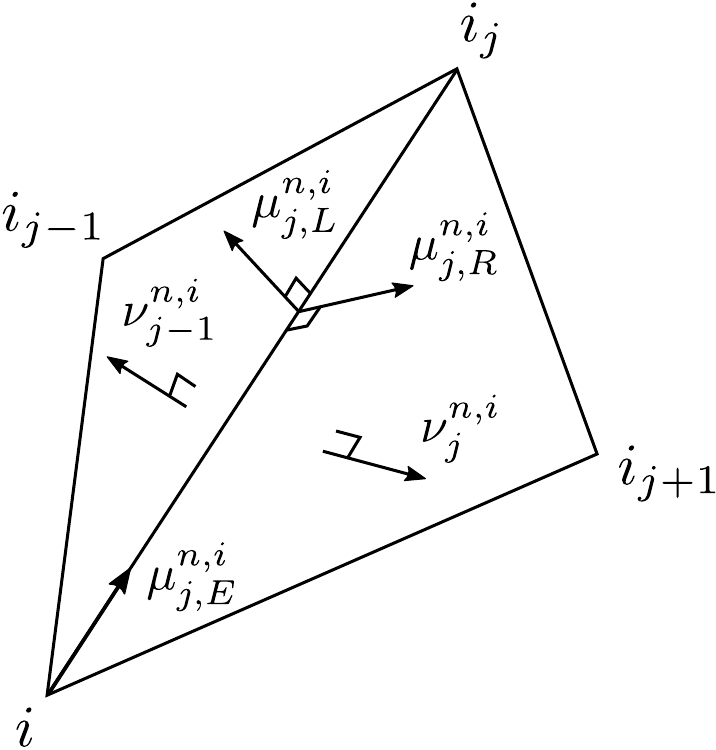}
  \caption{}
  \label{F:Pf_C_VcapT}
  \end{figure}

  %%%%%%%%%%%
  %%%%%%%%%%%
  \begin{comment}
\begin{figure}[tb]
  \centering
  \includegraphics[scale=0.8,clip]{HJ_Co-Normal.pdf}
  \caption{The normal and co-normal vectors}
  \label{F:Co-Normal}
\end{figure}
\end{comment}
%%%%%%%%%%%
  %%%%%%%%%%%

\noindent
Here, $\times$ denotes the vector product in $\mathbb{R}^3$.
Using the divergence theorem for integrals over a flat quadrilateral we have
  \begin{eqnarray*}
    \int_{V^{n,i}\cap T_j^{n,i}}\mathrm{div}_{\Gamma_h}p\,d\mathcal{H}^2  &= &  \int_{e_{j,R}^{n,i}}p\cdot\mu_{j,E}^{n,i}\,d\mathcal{H}^1+\int_{e_{j+1,L}^{n,i}}p\cdot\mu_{j+1,E}^{n,i}\,
    d\mathcal{H}^1  
+\int_{S_{j,L}^{n,i}}p\cdot\mu_{j,L}^{n,i}\,d\mathcal{H}^1+\int_{S_{j,R}^{n,i}}p\cdot\mu_{j+1,R}^{n,i}\,
d\mathcal{H}^1  \\
& = & p\cdot\{h_{j,R}^{n,i} \, \mu_{j,E}^{n,i}+h_{j+1,L}^{n,i} \, \mu_{j+1,E}^{n,i}+d^{n,i}(\mu_{j,L}^{n,i}+
\mu_{j+1,R}^{n,i})\},
  \end{eqnarray*}
since $ |e_{j,R}^{n,i}| = h_{j,R}^{n,i}$, $ |e_{j+1,L}^{n,i}| = h_{j+1,L}^{n,i}$ and  $ |S_{j,L}^{n,i}| = |S_{j,R}^{n,i}| 
= d^{n,i}$ by the definition of the volume $V^{n,i}$. 
  Summing up both sides of the above equality over $j=1,\dots,\Upsilon_i$ we obtain from (\ref{divp})
  \begin{align*}
0    = \sum_{j=1}^{\Upsilon_i}(p\cdot\mu^{n,i}_{j,E})(h_{j,L}^{n,i}+h_{j,R}^{n,i})+d^{n,i}\sum_{j=1}^{\Upsilon_i}p\cdot(\mu_{j,L}^{n,i}+\mu_{j,R}^{n,i}) 
    = -J_2+d^{n,i}\sum_{j=1}^{\Upsilon_i}\nabla_\Gamma\varphi(x_i^n)\cdot(\mu_{j,L}^{n,i}+\mu_{j,R}^{n,i}).
  \end{align*}
  Here the last line follows from $p=\nabla_\Gamma\varphi(x_i^n)$, \eqref{E:Co-normal_Edge}, and \eqref{Pf_C:Def_J1_J2}.   Hence
  \begin{align} \label{Pf_C:Estimate_J2}
 |   J_2 |  =|  d^{n,i}\sum_{j=1}^{\Upsilon_i}\nabla_\Gamma\varphi(x_i^n)\cdot(\mu_{j,L}^{n,i}+\mu_{j,R}^{n,i}) |
 \leq  C d^{n,i}  \|\nabla_\Gamma\varphi\|_{B(\overline{S_T})} \max_j | \mu_{j,L}^{n,i}+\mu_{j,R}^{n,i} |.
  \end{align}
%The vectors $\mu_{j,L}^{n,i}$ and $\mu_{j,R}^{n,i}$ are tangential to the triangles $T_j^{n,i}$ and $T_{j-1}^{n,i}$, %respectively.
%Moreover, they are normal to the edge $E_j^{n,i}$ and their norms are equal to one.
Note that, contrary to the case of a flat stationary domain considered in~\cite{KL15JCM}, the equality $\mu^{n,i}_{j,L}=-\mu_{j,R}^{n,i}$ does not hold in general because the triangles $T_{j-1}^{n,i}$ and $T_j^{n,i}$ do not lie in the same plane. Instead we  deduce from \eqref{E:Co-normal_Edge} and $|\mu_{j,E}^{n,i}|=1$ 
  \begin{eqnarray}
 \lefteqn{   |\mu_{j,L}^{n,i}+\mu_{j,R}^{n,i}| = |(\nu_{j}^{n,i}-\nu_{j-1}^{n,i})\times\mu_{j,E}^{n,i}| \leq |\nu_{j}^{n,i}-\nu_{j-1}^{n,i}| } \nonumber \\
    &\leq &  |\nu_j^{n,i}-\nu(x_i^n,t^n)|+|\nu(x_i^n,t^n)-\nu_{j-1}^{n,i}| \leq C h  \label{j2est}
  \end{eqnarray}
  by (\ref{nuest}). Inserting  \eqref{Pf_C:Estimate_J1}, \eqref{Pf_C:Estimate_J2} with (\ref{j2est})  into
  \eqref{Pf_C:Formula_I3} and taking into account \eqref{E:Quotient_EV} as well as  \eqref{E:Epsilon_Tau} 
  we derive
  \begin{eqnarray}
    |I_3| & \leq & C \frac{ \varepsilon^n_i}{| V^{n,i}|}\left( \sum_{j=1}^{\Upsilon_i} | E^{n,i}_j |^2 +
    d^{n,i} h \right) \left(\|\nabla_\Gamma \varphi\|_{B(\overline{S_T})}+\|\nabla_\Gamma^2\varphi\|_{B(\overline{S_T})}\right) \nonumber \\
    & \leq & Ch  \left(\|\nabla_\Gamma \varphi\|_{B(\overline{S_T})}+\|\nabla_\Gamma^2\varphi\|_{B(\overline{S_T})}\right). \label{Pf_C:Estimate_I3}
  \end{eqnarray}
%  and thus we apply \eqref{E:Quotient_Volume} and \eqref{E:Epsilon_Tau} to the right-hand side to get
%  \begin{align} \label{Pf_C:Estimate_I3}
%    |I_3| \leq \alpha_4C_1h\left(c_d\alpha_2+\frac{2c_\nu}{\tan\theta_0/2}\right)\left(\|\nabla_\Gamma \varphi\|%_{B(\overline{S_T})}+\|\nabla_\Gamma^2\varphi\|_{B(\overline{S_T})}\right),
%  \end{align}
%  where $\alpha_4$ is given by \eqref{E:Quotient_Constants}.
%  Finally, setting
%  \begin{align} \label{Pf_C:Const_3}
 %   C_3 := C_2+L_{H,2}\left(\frac{2c_d}{\sin^2\theta_0}+c_\nu\right)+\alpha_4C_1\left(c_d\alpha_2+\frac{2c_\nu}{\tan\theta_0/2}\right)
%  \end{align}
%  we observe that the inequality \eqref{E:Consistency} follows from the formula
%  \begin{align*}
%    \frac{\varphi_i^{n+1}-[S_h^n(I_h^n\varphi)]_i}{\tau}-\{\dot{\varphi}(x_i^n,t^n)+H(x_i^n,t^n,\nabla_\Gamma\varphi(x_i^n,t^n))\} = I_1+I_2+I_3
%  \end{align*}
The result now  follows from (\ref{i1i2i3}) together with   \eqref{Pf_C:Estimate_I1}, \eqref{Pf_C:Estimate_I2} and \eqref{Pf_C:Estimate_I3}.
\end{proof}

\begin{comment}
\begin{remark} It is possible to obtain monotonicity and consistency for the numerical Hamiltonian  
with a vertex--independent diffusion coefficient $\varepsilon^n_i=\varepsilon^n$, provided that (cf. \cite[Lemma 3]{LiYanChan03})
\begin{displaymath}
C_1  \max_{K \in \mathcal T_h(t^n)} h_K \leq \varepsilon^n \leq \tilde{C}_1  \max_{K \in \mathcal T_h(t^n)} h_K,
\; C_2  \varepsilon^n  \tau \leq \min_{K \in \mathcal T_h(t^n)} | K |.
\end{displaymath} 
\end{remark}
\end{comment}

%%%%%%%%%%%%%%%%%%%%%%%%%%%%% SECTION %%%%%%%%%%%%%%%%%%%%%%%%%%

\section{Convergence to viscosity solutions} \label{S:Convergence}
The purpose of this section is to prove that the approximate solution generated by the scheme 
\eqref{E:Scheme_Initial}--\eqref{E:Numerical_Hamiltonian} converges to a viscosity solution of 
the Hamilton--Jacobi equation \eqref{E:HJ_Equation} providing at the same time an existence result
for this problem. We start with a technical result that compares the nodal values of a solution
of the scheme  with those at the initial time, see Lemma 2.3 in \cite{KL15JCM}
for a similar result in the flat case.

\begin{lemma} \label{L:Nodal_FE_Push}
Suppose that $v^n_h=\sum_{i=1}^M v^n_i \chi(\cdot,t^n) \in V^n_h$ is a solution of $v^{n+1}_h=
S^n_h(v^n_h),n=0,\ldots,N-1$ with initial data $v^0_h(x^0_i)=v_0(x^0_i),i=1,\ldots,M$, where
$v_0:\Gamma(0) \rightarrow \mathbb{R}$ is Lipschitz continuous with constant $L_0$.
If  \eqref{E:Epsilon_Tau} holds,
  then there exists a constant $C_4>0$ depending  on $\gamma, \, H$ and $L_0$ such that
  \begin{align} \label{E:Nodal_FE_Push}
  \max_{i=1,\ldots,M}  |v_i^n-v_i^0| \leq C_4t^n, \quad n=0,1,\ldots,N.
  \end{align}
\end{lemma}
\begin{proof} Let us denote by $v_0^\sharp$  the push-forward of $v_0$ i.e.
$ v_0^\sharp(x,t) := v_0(\Phi^{-1}(x,t)), (x,t)\in\overline{S_T}$
and by $I^n_h v_0^\sharp \in V^n_h$ its interpolant.  Since $x^n_i=\Phi(x^0_i,t^n)$ we have
%Since $\Phi^{-1}$ is Lipschitz continuous on $\overline{S_T}$ we immediately infer from 
%(\ref{E:Lipschitz_Initial}) that 
%\begin{equation} \label{E:Lipschitz_Push}
% | u_0^\sharp(x,t)-u_0^\sharp(y,s)| \leq LL_{IF}(|x-y|+|t-s|) \quad  \forall (x,t),(y,s)\in\overline{S_T}.
%\end{equation}
% By \eqref{E:Def_Push} the nodal value of $I_h^nu_0^\sharp$ at $x_i^n$ is of the form
  \begin{align}   \label{Pf_NFP:Nodal_Inter_Push}
    [I_h^nv_0^\sharp]_i = I_h^nv_0^\sharp(x_i^n) = v_0^\sharp(x_i^n,t^n) = v_0(\Phi^{-1}(x_i^n,t^n))=v_0(x^0_i),
    \; i=1,\ldots,M.
  \end{align}
%  Applying $x_i^n=\Phi(x_i^0,t^n)$ to the last term we get
 % \begin{align} \label{Pf_NFP:Nodal_Inter_Push}
  %  [I_h^nu_0^\sharp]_i = u_0(x_i^0) = u_i^0 \quad\text{for all}\quad n=0,1,\dots,N,\,i=1,\dots,M.
 % \end{align}
  Note that the right-hand side is independent of $n$. We claim that there exists a constant
  $R\geq 0$ such that
  \begin{equation}  \label{nablabound}
  | \nabla_{\Gamma_h} I^n_h v_0^\sharp | \leq R \quad \mbox{ on } \Gamma_h(t^n).
  \end{equation}
  To see this, let us fix a triangle $K(t^n) \subset \Gamma_h(t^n)$ whose vertices are denoted for
  simplicity by $x^n_1,x^n_2$ and $x^n_3$. By transforming onto the unit triangle,  using
  (\ref{Pf_NFP:Nodal_Inter_Push}), the Lipschitz continuity of $v_0$ and $\Phi^{-1}$ as well as (\ref{E:Quotient_Diam_Local}) we obtain
  \begin{eqnarray*}
  | \nabla_{\Gamma_h} I^n_h v^\sharp_{0 | K(t^n)} | & \leq & \frac{C}{\rho_{K(t^n)}} 
  \max_{i=2,3} |   I^n_h v_0^\sharp(x^n_i)-
  I^n_h v_0^\sharp(x^n_1) | 
   =    \frac{C}{\rho_{K(t^n)}} \max_{i=2,3} | v_0(x^0_i) -v_0(x^0_1) | \\
   &  \leq &   \frac{C L_0}{\rho_{K(t^n)}} \max_{i=2,3} 
  | x^0_i -x^0_1 | 
   =  \frac{C L_0}{\rho_{K(t^n)}} \max_{i=2,3} | \Phi^{-1}(x^n_i,t^n) - \Phi^{-1}(x^n_1,t^n) | \\
   &  \leq &  \frac{C L_0}{\rho_{K(t^n)}} \max_{i=2,3} | x^n_i - x^n_1 | \leq  C L_0 \gamma =:R
   \end{eqnarray*}
proving (\ref{nablabound}).  Recalling the definition \eqref{E:Numerical_Hamiltonian} of the
numerical Hamiltonian we deduce with the help of (\ref{nablabound}) and \eqref{E:Epsilon_Tau} that
\begin{eqnarray}
\lefteqn{ \hspace{-1cm}
 |H_i^n([I_h^nv_0^\sharp]_i,[I_h^nv_0^\sharp]_{i_1},\dots,[I_h^nv_0^\sharp]_{i_{\Upsilon_i}})|   } \nonumber \\ 
&  \leq &   \sum_{j=1}^{\Upsilon_i}\frac{|V^{n,i}\cap T_j^{n,i}|}{|V^{n,i}|} \big| H\bigl(x_i^n,t^n,\nabla_{\Gamma_h}I_h^nv_0^\sharp|_{T_j^{n,i}}\bigr)  \big| + 
\frac{\varepsilon^n_i}{|V^{n,i}|}\sum_{j=1}^{\Upsilon_i}\frac{| [I_h^nv_0^\sharp]_{i_j}-[I_h^nv_0^\sharp]_{i} |}{|E_j^{n,i}|}(h_{j,L}^{n,i}+h_{j,R}^{n,i}) \nonumber  \\
& \leq & \max_{(x,t) \in \overline{S_T}, | p | \leq R} | H(x,t,p) | + C \,  \frac{\max_j (h^{n,i}_j)^2}{| V^{n,i} |}
 \leq C_4   \label{Pf_NFP:Estimate_Hamiltonian}
\end{eqnarray}
where $C_4$ can be chosen  independently of  $i$ and $n$.

\noindent
 Now let us show by induction with respect to $n=0,1,\dots,N$ that
  \begin{align} \label{Pf_NFP:Left_Ineq}
    v_i^n \leq [I_h^nv_0^\sharp]_i+C_4 t^n \quad\text{for all}\quad i=1,\dots,M.
  \end{align}
 Since $v^0_i=v_0(x^0_i)=[I^0_h v_0^\sharp]_i$ 
  the inequality \eqref{Pf_NFP:Left_Ineq} holds for $n=0$.
  Let us assume that \eqref{Pf_NFP:Left_Ineq} is true for some $n\in\{0,1,\dots,N-1\}$   
  so that $v_h^n\leq I_h^nv_0^\sharp+C_4t^n$ on $\Gamma_h(t^n)$.
  Applying  Lemma~\ref{L:Monotonicity} together with  \eqref{E:Invariance_Constant} we infer that 
  \begin{align*}
    v_h^{n+1} = S_h^n(v_h^n) \leq S_h^n(I_h^nv_0^\sharp+C_4t^n) = S_h^n(I_h^nv_0^\sharp)+C_4t^n
  \end{align*}
  on $\Gamma_h(t^{n+1})$, and hence   by \eqref{E:Scheme_Iteration}, \eqref{E:Numerical_Hamiltonian}, and \eqref{Pf_NFP:Estimate_Hamiltonian}
  \begin{eqnarray*}
    v_i^{n+1} & \leq &  [S_h^n(I_h^nv_0^\sharp)]_i+C_4t^n 
     =   [I_h^nv_0^\sharp]_i-\tau^n H_i^n([I_h^nv_0^\sharp]_i,[I_h^nv_0^\sharp]_{i_1},\dots,[I_h^nv_0^\sharp]_{i_{\Upsilon_i}}) + C_4 t^n \\
    &\leq &  [I_h^nv_0^\sharp]_i+ C_4 \tau^n +C_4 t^n = [I_h^{n+1}v_0^\sharp]_i+C_4t^{n+1}
  \end{eqnarray*}
for all $i=1,\dots,M$, where we used  \eqref{Pf_NFP:Nodal_Inter_Push} in the last step.
  Hence we see by induction that \eqref{Pf_NFP:Left_Ineq} holds for all $n=0,1,\dots,N$.
  By the same argument we can show that $[I_h^nv_0^\sharp]_i-C_4t^n\leq v_i^n$ for all $n=0,1,\dots,N$ and $i=1,\dots,M$.
  Finally, \eqref{Pf_NFP:Nodal_Inter_Push}, \eqref{Pf_NFP:Left_Ineq}, and the above inequality yield \eqref{E:Nodal_FE_Push}.
\end{proof}

\noindent
Let us denote by $u_h^n=\sum_{i=1}^Mu_i^n\chi_i(\cdot,t^n)\in V_h^n$, $n=0,1,\dots,N$  
the finite element function on $\Gamma_h(t^n)$ given by the numerical scheme \eqref{E:Scheme_Initial}--\eqref{E:Numerical_Hamiltonian}.
Now we define an approximate solution $u_h^l: \overline{S_T} \rightarrow \mathbb{R}$ by
\begin{align} \label{E:Approx_Sol}
  u_h^l(x,t) = \sum_{i=1}^Mu_i^n\chi_i^l(x,t), \quad t\in[t^n,t^{n+1}), \, x\in\Gamma(t)
\end{align}
for $n=0,1,\dots,N-1$ (we include $t=t^N=T$ when $n=N-1$), where $u_0$ is a given function on $\Gamma(0)$. 
%By definition, $u_h^l(\cdot,t)\in V_h^l(t)$ for all $t\in[0,T]$.
For $(x,t)\in\overline{S_T}$ set
\begin{align} \label{E:Approx_Limit}
  \bar{u}(x,t) := \limsup_{\substack{h\to0 \\ \overline{S_T}\ni(y,s)\to(x,t)}}u_h^l(y,s), \quad \underline{u}(x,t) := \liminf_{\substack{h\to0 \\ \overline{S_T}\ni(y,s)\to(x,t)}}u_h^l(y,s).
\end{align}
It follows from \cite[Section~V.2.1, Proposition~2.1]{BC97book} that
$\bar{u}\in USC(\overline{S_T})$ and $\underline{u}\in LSC(\overline{S_T})$. 
%by Lemma~\ref{L:FE_Limit_Semiconti} and thus $\bar{u}^\ast=\bar{u}$ and $\underline{u}_\ast=\underline{u}$ by %Lemma~\ref{L:Property_Envelope}.
Our aim is to show that $\bar{u}$ (resp. $\underline{u}$) is a subsolution (resp. supersolution) to \eqref{E:HJ_Equation}. As a first step we prove

\begin{lemma} \label{L:Limit_Initial}
  Let $\bar{u}$ and $\underline{u}$ be given by \eqref{E:Approx_Sol}--\eqref{E:Approx_Limit}.
  Assume that  \eqref{E:Epsilon_Tau} is satisfied and that $u_0\in C(\Gamma(0))$.
  Then $\bar{u}(\cdot,0)=\underline{u}(\cdot,0)=u_0$ on $\Gamma(0)$.
\end{lemma}

\begin{proof}
  Fix $x_0\in\Gamma(0)$.
  By \eqref{E:Approx_Limit} it immediately follows that $\underline{u}(x_0,0)\leq\bar{u}(x_0,0)$.
  Therefore, if the inequality
  \begin{align} \label{Pf_LI:Goal}
    \bar{u}(x_0,0) \leq u_0(x_0) \leq \underline{u}(x_0,0)
  \end{align}
  holds, then we get $\bar{u}(x_0,0)=\underline{u}(x_0,0)=u_0(x_0)$.
  Let us prove \eqref{Pf_LI:Goal}. 
  Since $\Gamma(0)$ is compact in $\mathbb{R}^3$, the function $u_0\in C(\Gamma(0))$ is bounded and uniformly continuous on $\Gamma(0)$.
  Hence setting
  \begin{align*}
    \omega_0(r) := \sup\{|u_0(x)-u_0(x_0)| \mid x\in\Gamma(0),\,|x-x_0|\leq r\}, \quad r\in[0,\infty),
  \end{align*}
  we see that $\omega_0(0)=0$ and $\omega_0$ is bounded, nondecreasing, continuous at $r=0$.
  From this fact and the proof of~\cite[Lemma~2.1.9 (i)]{G06book} there exists a bounded, nondecreasing, and continuous function $\omega$ on $[0,\infty)$ satisfying $\omega(0)=0$ and $\omega_0\leq \omega$ on $[0,\infty)$.
  Fix an arbitrary $\delta>0$.
  By the above properties of $\omega$ we may take a constant $A_\delta>0$ such that $\omega(r)\leq\delta+A_\delta r^2$ for all $r\in[0,\infty)$.
  From this inequality and $|u_0(x)-u_0(x_0)|\leq\omega_0(|x-x_0|)\leq \omega(|x-x_0|)$ it follows that
  \begin{align} \label{Pf_LI:Compare_UV_Conti}
    u_0(x) \leq u_0(x_0)+\delta+A_\delta|x-x_0|^2 \quad\text{ for all } x\in\Gamma(0).
  \end{align}
  Now we construct $v_h^n=\sum_{i=1}^Mv_i^n\chi_i(\cdot,t^n)\in V_h^n$, $n=0,1,\dots,N$ by \eqref{E:Scheme_Initial}--\eqref{E:Numerical_Hamiltonian} from the initial value $v_0(x):=A_\delta|x-x_0|^2$, $x\in\Gamma(0)$.
  Then by interpolating both sides of \eqref{Pf_LI:Compare_UV_Conti} on $\Gamma_h(0)$ and observing that $u_0(x_0)+\delta$ is constant we have
  \begin{align*}
    u_h^0 \leq u_0(x_0)+\delta+v_h^0 \quad\text{on}\quad \Gamma_h(0).
  \end{align*}
Combining this inequality with Lemma~\ref{L:Monotonicity} and \eqref{E:Invariance_Constant} we obtain
  \begin{align*}
    u_h^1 = S_h^0(u_h^0) &\leq S_h^0(u_0(x_0)+\delta+v_h^0) 
    = u_0(x_0)+\delta+S_h^0(v_h^0) = u_0(x_0)+\delta+v_h^1 \quad\text{on}\quad \Gamma_h(t^1)
  \end{align*}
 and then inductively  $u_h^n\leq u_0(x_0)+\delta+v_h^n$ on $\Gamma_h(t^n)$ for $n=0,1,\dots,N$, or
  \begin{align} \label{Pf_LI:Compare_UV_0i}
     u_i^n \leq u_0(x_0)+\delta+v_i^n \leq u_0(x_0)+\delta+v_i^0+C_4t^n \quad 
  \end{align}
 for $n = 0,1,\dots,N,\, i = 1,\dots,M$, where we applied Lemma \ref{L:Nodal_FE_Push} for $v^n_h$.
  Multiplying by $\chi_i^l(\cdot,t), t \in [t^n,t^{n+1})$ and summing over $i=1,\ldots,M$ we infer with the help of 
  (\ref{Pf_NFP:Nodal_Inter_Push}) (with $t$ instead of $t^n$)
    \begin{align} \label{Pf_LI:Compare_UV_Lift}
    u_h^l(x,t) \leq u_0(x_0)+\delta+[I_h^tv_0^\sharp]^l(x)+C_4t \quad\text{for all}\quad (x,t)\in\overline{S_T}.
  \end{align}
%  Since $u_h^l$ is given by \eqref{E:Approx_Sol} and $[I_h^tv_0^\sharp]^l=\sum_{i=1}^Mv_0^i\chi_i^l(\cdot,t)$ on %$\Gamma(t)$, $t\in[0,T]$ by
%  \begin{align*}
 %   v_0^\sharp(x_i(t),t) = v_0(\Phi^{-1}(\Phi(x_i^0,t),t)) = v_0(x_i^0) = v_i^0,
 % \end{align*}
 % from \eqref{Pf_LI:Compare_UV_0i} and $t^n\leq t$ for $t\in[t^n,t^{n+1})$, $n=0,1,\dots,N$ it follows that
\noindent
Since $v_0^\sharp(x_0,0)=v_0(x_0)=0$ and $v_0^\sharp$ is Lipschitz continuous on $\overline{S_T}$
we may estimate
  \begin{eqnarray*}
  \lefteqn{
    |[I_h^tv_0^\sharp]^l(x)|  \leq    |[I_h^tv_0^\sharp]^l(x) - v_0^\sharp(x,t) | + | v_0^\sharp(x,t) -
    v_0^\sharp(x_0,0) | }  \\
    & \leq & \Vert v_0^\sharp(\cdot,t) - [I_h^tv_0^\sharp]^l \Vert_{B(\Gamma(t))} + C(| x-x_0| +t) 
    \leq  C(h + | x-x_0| +t),
  \end{eqnarray*}
  where we also used Lemma \ref{L:Interpolation}. Combining this estimate with  (\ref{Pf_LI:Compare_UV_Lift})
  we infer
  \begin{displaymath}
 \bar{u}(x_0,0) =   \limsup_{\substack{h\to0 \\ \overline{S_T}\ni(x,t)\to(x_0,0)}}u_h^l(x,t)
 \leq u_0(x_0) + \delta.
 \end{displaymath}
  Since $\delta>0$ is arbitrary, it follows that $\bar{u}(x_0,0)\leq u_0(x_0)$.
  By the same argument we can show $u_0(x_0)\leq \underline{u}(x_0,0)$.
  Hence \eqref{Pf_LI:Goal} is valid and the lemma follows.
\end{proof}

\begin{lemma} \label{L:Limit_Subsol}
  Under the same assumptions as in Lemma~\ref{L:Limit_Initial}, $\bar{u}$ (resp. $\underline{u}$) is a 
  subsolution (resp. supersolution) to \eqref{E:HJ_Equation}.
\end{lemma}
\begin{proof} We know from Lemma \ref{L:Limit_Initial} that $\bar{u}(x,0)=\underline{u}(x,0)=u_0(x),
x \in \Gamma(0)$ so that it remains to verify \eqref{E:Ineq_Subsol}. \\
  Let us suppose first that $\varphi\in C^2(\overline{S_T})$ and that  $\bar{u}-\varphi$ takes a local maximum 
  at $(x_0,t_0)\in\overline{S_T}$ with $t_0>0$. Since $\bar{u}$ is bounded on $\overline{S_T}$ we may
  assume by a standard argument that $\bar{u} - \varphi$ has a strict global maximum at
  $(x_0,t_0)$.
  Let $\varphi_h^l$ be given by 
    \begin{align} \label{E:FE_Test}
    \varphi_h^l(x,t) := \sum_{i=1}^M\varphi_i^n\chi_i^l(x,t), \quad t\in[t^n,t^{n+1}),\,x\in\Gamma(t),
  \end{align}
  where $\varphi_i^n:=\varphi(x_i^n,t^n)$, $i=1,\dots,M$ and we include $t=t^N=T$ if $n=N-1$.
 We claim that 
  \begin{align} \label{Pf_LS:Limsup_UPhi}
    (\bar{u}-\varphi)(x,t) = \limsup_{\substack{h\to0 \\ \overline{S_T}\ni(y,s)\to(x_0,t_0)}}(u_h^l-\varphi_h^l)(y,s).
  \end{align}
  In order to see this, we note that in view of the Lipschitz continuity of $\varphi$  on $\overline{S_T}$ 
  it is sufficient to show that
  $\varphi_h^l \rightarrow \varphi$ uniformly on $\overline{S_T}$. But,
  \begin{eqnarray*}
  \Vert \varphi^l_h - \varphi \Vert_{B(\overline{S_T})}  
 &  \leq &   \sup_{t \in [0,T]} \Vert \varphi(\cdot,t) -
  [I^t_h \varphi]^l \Vert_{B(\Gamma(t))}  \\
  &&  + \max_{n=0,\ldots,N-1} \sup_{x \in \Gamma(t),t^n \leq t \leq
  t^{n+1}} \big| \sum_{i=1}^M ( \varphi(x_i(t),t) - \varphi(x^n_i,t^n)) \chi^l_i(x,t) \big|  \\
 & \leq & Ch +  \max_{n=0,\ldots,N-1} \sup_{i=1,\ldots,M, t^n \leq t \leq
  t^{n+1}} |  \varphi(x_i(t),t) - \varphi(x^n_i,t^n) | 
   \leq  C(h+ \tau^n) \leq Ch
 \end{eqnarray*}  
   by Lemma \ref{L:Interpolation}, the fact that $x_i(t)=\Phi(x^0_i,t)$, the Lipschitz continuity of
   $\varphi$ and $\Phi$ as well as \eqref{E:Epsilon_Tau}. Thus, (\ref{Pf_LS:Limsup_UPhi}) holds so that
  there exist $h_j>0$ and $(y_j,s_j)\in\overline{S_T}$, $j\in\mathbb{N}$ with $h_j\to0$, $(y_j,s_j)\to(x_0,t_0)$, and $(u_{h_j}^l-\varphi_{h_j}^l)(y_j,s_j)\to(\bar{u}-\varphi)(x_0,t_0)$ as $j\to\infty$.
  For each $j\in\mathbb{N}$, the function $u_{h_j}^l-\varphi_{h_j}^l$ is of the form
  \begin{align*}
    (u_{h_j}^l-\varphi_{h_j}^l)(x,t) = \sum_{i=1}^M(u_i^n-\varphi_i^n)\chi_i^l(x,t), \; \; 
    x\in\Gamma(t), t\in[t^n,t^{n+1}),n=0,\dots,N-1.
  \end{align*}
Let us choose  $n_j\in\{0,1,\dots,N\}$ and $i_j\in\{1,\dots,M\}$ such that
  \begin{align*}
    u_{i_j}^{n_j}-\varphi_{i_j}^{n_j}=\max\{u_i^n-\varphi_i^n \mid n=0,\dots,N,\,i=1,\dots,M\}
  \end{align*}
  and use $\chi_i(x,t)\geq0$, $i=1,\dots,M$ and $\sum_{i=1}^M\chi_i^l(x,t)=1$ to get
  \begin{equation}  \label{Pf_LS:Max_UPhi_Hj}
      (u_{h_j}^l-\varphi_{h_j}^l)(x,t) \leq (u_{i_j}^{n_j}-\varphi_{i_j}^{n_j})\sum_{i=1}^M\chi_i^l(x,t) 
     = (u_{h_j}^l-\varphi_{h_j}^l)(x_{i_j}^{n_j},t^{n_j})
  \end{equation}
  for all $(x,t)\in\overline{S_T}$.
  In particular, for all $j\in\mathbb{N}$,
  \begin{align*}
    (u_{h_j}^l-\varphi_{h_j}^l)(y_j,s_j) \leq (u_{h_j}^l-\varphi_{h_j}^l)(x_{i_j}^{n_j},t^{n_j}).
  \end{align*}
  Since $(x_{i_j}^{n_j},t^{n_j})$ belongs to the compact set $\overline{S_T}$, we may assume 
  (up to a subsequence) that there exists $(\bar{x},\bar{t})\in\overline{S_T}$ such that $(x_{i_j}^{n_j},t^{n_j})\to(\bar{x},\bar{t})$ as $j\to\infty$.
  Then by the above inequality and \eqref{Pf_LS:Limsup_UPhi} we have
  \begin{displaymath}
   (\bar{u}-\varphi)(x_0,t_0) = \lim_{j\to\infty}(u_{h_j}^l-\varphi_{h_j}^l)(y_j,s_j) 
    \leq  \limsup_{j\to\infty}(u_{h_j}^l-\varphi_{h_j}^l)(x_{i_j}^{n_j},t^{n_j}) 
    \leq (\bar{u}-\varphi)(\bar{x},\bar{t})
  \end{displaymath}
  where the last inequality follows from  the fact that $\bar u - \varphi \in USC(\overline{S_T)}$. 
  Recalling that $\bar{u}-\varphi$ takes a strict global  maximum at $(x_0,t_0)$ we infer that
  $(\bar{x},\bar{t})=(x_0,t_0)$. In particular, since $\lim_{j\to\infty}t^{n_j}=\bar{t}=t_0>0$ we have 
  for sufficiently large $j$ that $t^{n_j}>0$ i.e. $n_j\geq 1$.
  Thus we can set $(x,t)=(x_i^{n_j-1},t^{n_j-1})$ in \eqref{Pf_LS:Max_UPhi_Hj} to obtain
  \begin{align*}
    (u_{h_j}^l-\varphi_{h_j}^l)(x_i^{n_j-1},t^{n_j-1}) \leq \delta_j := u_{i_j}^{n_j}-\varphi_{i_j}^{n_j},
  \end{align*}
  or equivalently, $u_i^{n_j-1}\leq \varphi_i^{n_j-1}+\delta_j$ for $i=1,\dots,M$.
  From this we see that
  \begin{align*}
    u_{h_j}^{n_j-1} \leq I_{h_j}^{n_j-1}\varphi+\delta_j \quad\text{on}\quad \Gamma_{h_j}(t^{n_j-1}),
  \end{align*}
 % where $I_h^n\varphi=\sum_{i=1}^M\varphi_i^n\chi_i(\cdot,t^n)\in V_h^n$ for $h>0$ and $n=0,1,\dots,N$.
and   then by Lemma~\ref{L:Monotonicity} and  \eqref{E:Invariance_Constant}
  \begin{align*}
    u_{h_j}^{n_j} = S_{h_j}^{n_j-1}(u_{h_j}^{n_j-1}) \leq S_{h_j}^{n_j-1}(I_{h_j}^{n_j-1}\varphi+\delta_j) = S_{h_j}^{n_j-1}(I_{h_j}^{n_j-1}\varphi)+\delta_j \; \mbox{ on } \Gamma_{h_j}(t^{n_j}).
  \end{align*}
 Inserting  $x=x_{i_j}^{n_j}\in\Gamma_{h_j}(t^{n_j})$ into this inequality we get
  \begin{align*}
    u_{i_j}^{n_j} \leq [S_{h_j}^{n_j-1}(I_{h_j}^{n_j-1}\varphi)]_{i_j}+\delta_j = [S_{h_j}^{n_j-1}(I_{h_j}^{n_j-1}\varphi)]_{i_j}+u_{i_j}^{n_j}-\varphi_{i_j}^{n_j}
  \end{align*}
%  Here $[S_{h_j}^{n_j-1}(I_{h_j}^{n_j-1}\varphi)]_{i_j}$ is the nodal value of $S_{h_j}^{n_j-1}(I_{h_j}^{n_j-1}\varphi)$ at $x_{i_j}^{n_j}$.
by the definition of $\delta_j$ and hence,
  \begin{align} \label{Pf_LS:Frac_Phij}
   \varphi_{i_j}^{n_j}-[S_{h_j}^{n_j-1}(I_{h_j}^{n_j-1}\varphi)]_{i_j} \leq 0.
  \end{align}
  Since $\varphi\in C^2(\overline{S_T})$, we can combine (\ref{Pf_LS:Frac_Phij})  with Lemma
  \ref{L:Consistency} to derive
  \begin{align} \label{Pf_LS:Ineq_Sub_Discrete}
    \partial^\bullet\varphi(x_{i_j}^{n_j-1},t^{n_j-1})+H(x_{i_j}^{n_j-1},t^{n_j-1},\nabla_\Gamma\varphi(x_{i_j}^{n_j-1},t^{n_j-1})) \leq C_\varphi h_j
  \end{align}
 and observing that 
  \begin{displaymath}
  | (x_{i_j}^{n_j},t^{n_j}) - (x_{i_j}^{n_j-1},t^{n_j-1}) | \leq C \tau^{n_j-1}
  \leq C h_j \rightarrow 0, \, j 
  \rightarrow \infty
  \end{displaymath}
we obtain    \eqref{E:Ineq_Subsol} by sending $j\to\infty$ in \eqref{Pf_LS:Ineq_Sub_Discrete}.

  %%%%%%%%%%%
  %%%%%%%%%%
 \begin{comment} 
  Let $j\to\infty$ in \eqref{Pf_LS:Ineq_Sub_Discrete} (note that the behavior of $n_j$ and $i_j$ themselves as $j\to\infty$ is unknown).
  By $t^{n_j}-t^{n_j-1}=\tau$, \eqref{E:Epsilon_Tau} and $\lim_{j\to\infty}t^{n_j}=t_0$,
  \begin{align*}
    |t_0-t^{n_j-1}| \leq |t_0-t^{n_j}|+C_2h_j \to 0 \quad\text{as}\quad j\to0.
  \end{align*}
  Also, by $x_i^n=\Phi(x_i^0,t^n)$ and \eqref{E:Flow_Map},
  \begin{align*}
    |x_{i_j}^{n_j}-x_{i_j}^{n_j-1}| &= \left|\int_{t^{n_j-1}}^{t^{n_j}}\frac{\partial\Phi}{\partial s}(x_{i_j}^0,s)\,ds\right| = \left|\int_{t^{n_j-1}}^{t^{n_j}}v_\Gamma(\Phi(x_{i_j}^0,s),s)\,ds\right| \\
    &\leq \|v_\Gamma\|_{B(\overline{S_T})}|t^{n_j}-t^{n_j-1}|.
  \end{align*}
  From this inequality, $|t^{n_j}-t^{n_j-1}|=\tau\leq C_2h_j$, and $\lim_{j\to\infty}x_{i_j}^{n_j}=x_0$,
  \begin{align*}
    |x_0-x_{i_j}^{n_j-1}| \leq |x_0-x_{i_j}^{n_j}|+C_2\|v_\Gamma\|_{B(\overline{S_T})}h_j \to 0 \quad\text{as}\quad j\to \infty.
  \end{align*}
  \end{comment}
%%%%%%%%%%%%%%%
%%%%%%%%%%%%%%%%

\noindent
  Finally,  let $\varphi\in C^1(\overline{S_T})$ and suppose that $\bar{u}-\varphi$ takes a local maximum at $(x_0,t_0)\in\overline{S_T}$, $t_0>0$.
  As in the first part of the proof, we may assume that $\bar{u}-\varphi$ takes a strict global  maximum at $(x_0,t_0)$. 
  Let us approximate $\varphi$ by a sequence $(\varphi_{\delta}) \subset C^2(\overline{S_T)}$
  such that $\varphi_{\delta} \rightarrow \varphi$ in $C^1(\overline{S_T})$ as $\delta \rightarrow 0$.
%  We extend $\varphi$ (first to $N_T$ and then by multiplying it by a smooth cut-off function supported in $N_T$) to a compactly supported $C^1$ function on $\mathbb{R}^4$ and mollify it to get a smooth function $\varphi_\delta$ on $\mathbb{R}^4$, $\delta>0$ such that
 % \begin{align} \label{Pf_LS:Mollify}
 %   \begin{aligned}
 %     \lim_{\delta\to0}\|\varphi-\varphi_\delta\|_{B(\overline{S_T})} &= \lim_{\delta\to0}\|\partial^\bullet\varphi-\partial^\bullet\varphi_\delta\|_{B(\overline{S_T})} \\
 %     &= \lim_{\delta\to0}\|\nabla_\Gamma\varphi-\nabla_\Gamma\varphi_\delta\|_{B(\overline{S_T})} = 0.
 %   \end{aligned}
%  \end{align}
For a suitable subsequence  there exist $(x_\delta,t_\delta)\in\overline{S_T}$ such that 
$ (x_\delta,t_\delta) \rightarrow (x_0,t_0)$ and 
 $\bar{u}-\varphi_\delta$ takes a global maximum at  $(x_\delta,t_\delta)$.
 In particular, $t_\delta>0$ for sufficiently small $\delta>0$.
It follows from the first part of the proof that
  \begin{align*}
    \partial^\bullet\varphi_\delta(x_\delta,t_\delta)+H(x_\delta,t_\delta,\nabla_\Gamma\varphi_\delta(x_\delta,t_\delta)) \leq 0.
  \end{align*}
  Letting $\delta\to0$ in the above inequality 
%  and using $\lim_{\delta\to0}(x_\delta,t_\delta)=(x_0,t_0)$, \eqref{E:Hamiltonian_Lip_xt}, \eqref{E:Hamiltonian_Lip_p}, and \eqref{Pf_LS:Mollify} 
we see that $\varphi$ satisfies \eqref{E:Ineq_Subsol} at $(x_0,t_0)$, so that 
 $\bar{u}$ is a subsolution to \eqref{E:HJ_Equation}. In the same way one shows that 
 $\underline{u}$ is a supersolution.
\end{proof}

\noindent
Finally, let us prove the existence of a viscosity solution to \eqref{E:HJ_Equation}.

\begin{theorem} \label{T:Viscosity_Sol}
Suppose that $u_0 \in C(\Gamma(0))$. Then
 there exists a unique viscosity solution to \eqref{E:HJ_Equation}.
\end{theorem}

\begin{proof}
  The uniqueness of a viscosity solution was already shown in Corollary~\ref{C:Uniqueness_Solution}.
  Let us prove the existence.
  Since $u_0\in C(\Gamma(0))$, Lemmas~\ref{L:Limit_Initial} and~\ref{L:Limit_Subsol} imply that $\bar{u}$ and $\underline{u}$ constructed by \eqref{E:Approx_Sol}--\eqref{E:Approx_Limit} are a subsolution and supersolution to \eqref{E:HJ_Equation}, respectively, and satisfy $\bar{u}(\cdot,0)=\underline{u}(\cdot,0)=u_0$ on $\Gamma(0)$.
  Hence we can apply the comparison principle (see Theorem~\ref{L:Comparison_Principle}) to the subsolution $\bar{u}$ and the supersolution $\underline{u}$ to get $\bar{u}\leq\underline{u}$ on $\overline{S_T}$.
  Moreover, by \eqref{E:Approx_Limit} we easily see that $\underline{u}\leq\bar{u}$ on $\overline{S_T}$.
  Therefore, $u:=\bar{u}=\underline{u}$ is a viscosity solution to \eqref{E:HJ_Equation}.
\end{proof}

%%%%%%%%%%%%%%%%%%%%%%%%%% SECTION %%%%%%%%%%%%%%%%%%%%%%
\section{Error bound} \label{T:Error_Bound}

\noindent
In this section we prove an error estimate between the viscosity solution to \eqref{E:HJ_Equation} and the numerical solution given by the scheme \eqref{E:Scheme_Initial}--\eqref{E:Numerical_Hamiltonian}.

\begin{theorem} \label{T:Error_Estimate}
Suppose that the viscosity solution $u$ of \eqref{E:HJ_Equation}  is Lipschitz continuous on 
$\overline{S_T}$ in the sense that
  \begin{align} \label{E:Lipschitz_Visc_Sol}
    |u(x,t)-u(y,s)| \leq L_U(|x-y|+|t-s|)
  \end{align}
  for all $(x,t),(y,s)\in\overline{S_T}$, where $L_U>0$ is a constant independent of $(x,t)$ and $(y,s)$.
  Assume further that  \eqref{E:Epsilon_Tau} is  satisfied and denote 
  by $u_h^n=\sum_{i=1}^Mu_i^n\chi_i(\cdot,t^n)\in V_h^n$  the finite element function constructed from $u_0$ using \eqref{E:Scheme_Initial}--\eqref{E:Numerical_Hamiltonian}.
Then there exist $h_0>0$ and  a constant $C>0$ independent of $h$ such that
  \begin{align} \label{E:Error_Discrete}
    \max_{1\leq i\leq M,\,0\leq n\leq N}|u(x_i^n,t^n)-u_i^n| \leq Ch^{1/2} \quad\text{for all}\quad h\in(0,h_0).
  \end{align}
%  Also, let $u_h^l\colon\overline{S_T}\to\mathbb{R}$ be the numerical solution given by \eqref{E:Approx_Sol}.
 % Then under the same assumptions as above there exists a constant $\widetilde{C}_E>0$ independent of $h$ such that
%  \begin{align} \label{E:Error_Continuous}
 %   \|u-u_h^l\|_{B(\overline{S_T})} \leq \widetilde{C}_Eh^{1/2} \quad\text{for all}\quad h\in(0,1).
 % \end{align}
\end{theorem}

\begin{proof} 
  The argument is similar to that in the proof of the comparison principle (see Theorem~\ref{L:Comparison_Principle}). Let us define
  \begin{align} \label{Pf_EE:Def_Phi}
    \Psi(x,t,i,n) := u(x,t) - \rho \sqrt{h}  \, t -u_i^n-\frac{|x-x_i^n|^2+|t-t^n|^2}{\sqrt{h}}
  \end{align}
  for $(x,t)\in\overline{S_T}, \, i \in \lbrace 1,\dots,M \rbrace$ and $n \in \lbrace 0,1,\dots,N \rbrace$. Here, the constant $\rho>0$ is subject
  to $\rho \sqrt{h} \leq 1$ and will be chosen later.
 % For each $i=1,\dots,M$ and $n=0,1,\dots,N$, $\Psi(\cdot,\cdot,i,n)$ is continuous on $\overline{S_T}$.
 Clearly, 
 \begin{eqnarray}
 \lefteqn{  \hspace{-1.5cm}
 \max_{1\leq i\leq M,\,0\leq n\leq N} (u(x_i^n,t^n)-u_i^n) =  \max_{1\leq i\leq M,\,0\leq n\leq N}
[ \Psi(x^n_i,t^n,i,n)+ \rho \sqrt{h} \, t^n ] } \nonumber \\
& \leq & \max_{(x,t)\in\overline{S_T},\,i=1,\dots,M,\,n=0,\dots,N} \Psi(x,t,i,n) + \rho \sqrt{h} \, T 
 =   \Psi(x_0,t_0,i_0,n_0) + \rho \sqrt{h} \, T
\label{maxbound}
\end{eqnarray}
for some $(x_0,t_0)\in\overline{S_T}$, $i_0\in\{1,\dots,M\}$ and $n_0\in\{0,1,\dots,N\}$.
 In particular,  we have $\Psi(x_{i_0}^{n_0},t^{n_0},i_0,n_0)\leq\Psi(x_0,t_0,i_0,n_0)$, i.e.
  \begin{displaymath}
    u(x_{i_0}^{n_0},t^{n_0})- \rho \sqrt{h} \, t^{n_0} -u_{i_0}^{n_0}  
    \leq  u(x_0,t_0)- \rho \sqrt{h} \, t_0 -u_{i_0}^{n_0}    -\frac{|x_0-x_{i_0}^{n_0}|^2+|t_0-t^{n_0}|^2}{\sqrt{h}}.
  \end{displaymath}
  From this, \eqref{E:Lipschitz_Visc_Sol}, and the fact that $\rho \sqrt{h} \leq 1$  it follows that
  \begin{eqnarray*}
  \lefteqn{ \hspace{-1cm}
    \frac{|x_0-x_{i_0}^{n_0}|^2+|t_0-t^{n_0}|^2}{\sqrt{h}} \leq u(x_0,t_0)-u(x_{i_0}^{n_0},t^{n_0}) +
    \rho \sqrt{h} (t^{n_0}-t_0) } \\
    &\leq &  L_U(|x_0-x_{i_0}^{n_0}|+|t_0- t^{n_0}|)+|t_0-t^{n_0}| 
    \leq C (|x_0-x_{i_0}^{n_0}|^2+|t_0-t^{n_0}|^2)^{1/2}
  \end{eqnarray*} 
  and hence 
 \begin{align} \label{Pf_EE:Const_1}
      \frac{(|x_0-x_{i_0}^{n_0}|^2+|t_0-t^{n_0}|^2)^{1/2}}{\sqrt{h}} \leq C.
  \end{align}  
Now let us consider several possible cases.  \\
   \textbf{Case 1:} $t_0>0$ and $n_0\geq1$. By exploiting the fact that $u$ is a subsolution we 
   obtain as in   (\ref{Pf_CP:Subsol_Test}) 
    \begin{equation} \label{Pf_EE:Case1_Ineq_Sub}
   \frac{2}{\sqrt{h}}(t_0-t^{n_0})+\frac{2}{\sqrt{h}}v_\Gamma(x_0,t_0)\cdot(x_0-x_{i_0}^{n_0}) 
    +H \bigl( x_0,t_0,\frac{2}{\sqrt{h}}P_\Gamma(x_0,t_0)(x_0-x_{i_0}^{n_0}) \bigr) \leq - \rho \sqrt{h}.
  \end{equation} 
 On the other hand, since $\Psi(x_0,t_0,i,n_0-1) \leq \Psi(x_0,t_0,i_0,n_0), i=1,\ldots,M$ we infer 
  \begin{displaymath}
  \varphi^{n_0-1}_i - u^{n_0-1}_i \leq \varphi^{n_0}_{i_0} - u^{n_0}_{i_0}, \quad i=1,\ldots,M,
  \end{displaymath}
  where
  \begin{displaymath}
  \varphi^n_i=\varphi(x^n_i,t^n) \quad \mbox{ and } \quad \varphi(x,t)=- \frac{| x_0-x|^2 + (t_0-t)^2}{
  \sqrt{h}}.
  \end{displaymath}
  Hence,  $I_h^{n_0-1}\varphi \leq u_h^{n_0-1}+ \varphi^{n_0}_{i_0} - u^{n_0}_{i_0}$ on $\Gamma_h(t^{n_0-1})$
  so that we deduce with the help of Lemma~\ref{L:Monotonicity}, \eqref{E:Invariance_Constant}
  and the definition of the scheme
  \begin{displaymath}
  S_h^{n_0-1} (I_h^{n_0-1}\varphi)  \leq S_h^{n_0-1} (u^{n_0-1}_h) +  \varphi^{n_0}_{i_0} - u^{n_0}_{i_0}
  = u^{n_0}_h +  \varphi^{n_0}_{i_0} - u^{n_0}_{i_0}.
  \end{displaymath}
  Evaluting the above inequality for $x=x^{n_0}_{i_0}$ we find that
  \begin{displaymath}
  [ S_h^{n_0-1} (I_h^{n_0-1}\varphi)]_{i_0} \leq \varphi^{n_0}_{i_0},
  \end{displaymath}
 from which we infer that 
 \begin{equation} \label{supervarphi}
  -\partial^{\bullet} \varphi(x^{n_0}_{i_0},t^{n_0}) - H \bigl( x^{n_0}_{i_0},t^{n_0},
 \nabla_{\Gamma}  \varphi(x^{n_0}_{i_0},t^{n_0}) \bigr) \leq A+B,
 \end{equation}
 where
 \begin{eqnarray*}
 A& = &- \partial^{\bullet} \varphi(x^{n_0-1}_{i_0},t^{n_0-1}) - H \bigl( x^{n_0-1}_{i_0},t^{n_0-1},
 \nabla_{\Gamma}  \varphi(x^{n_0-1}_{i_0},t^{n_0-1}) \bigr)  \\
 &&  + \frac{\varphi^{n_0}_{i_0} - [ S_h^{n_0-1} (I_h^{n_0-1}\varphi)]_{i_0}}{\tau^{n_0-1}},  \\
B & = &  [\partial^{\bullet} \varphi(x^{n_0-1}_{i_0},t^{n_0-1}) - \partial^{\bullet} \varphi(x^{n_0}_{i_0},t^{n_0})] \\
&& +[ H \bigl( x^{n_0-1}_{i_0},t^{n_0-1},
 \nabla_{\Gamma}  \varphi(x^{n_0-1}_{i_0},t^{n_0-1}) \bigr)-H \bigl( x^{n_0}_{i_0},t^{n_0},
 \nabla_{\Gamma}  \varphi(x^{n_0}_{i_0},t^{n_0}) \bigr)].
 \end{eqnarray*}
 We deduce  from Lemma \ref{L:Consistency} that
 \begin{equation}  \label{Aest}
 | A | \leq C_3h\left(\|\nabla_\Gamma\varphi \|_{B(\overline{S_T})}+\|\nabla_\Gamma^2
 \varphi \|_{B(\overline{S_T})}+\| (\partial^\bullet )^2 \varphi\|_{B(\overline{S_T})}\right) \leq C \sqrt{h}
 \end{equation}
 since
 \begin{eqnarray}
 \partial^\bullet \varphi(x,t) &= & -\frac{2}{\sqrt{h}}(t-t_0)- \frac{2}{\sqrt{h}} v_{\Gamma}(x,t) \cdot (x-x_0),
  \label{dervarphi1} \\
  \nabla_{\Gamma} \varphi(x,t) & = & - \frac{2}{\sqrt{h}} P_{\Gamma}(x,t) (x-x_0). \label{dervarphi2}
 \end{eqnarray}
 Using (\ref{dervarphi1}), (\ref{dervarphi2}), (\ref{E:Hamiltonian_Lip_xt}), (\ref{E:Hamiltonian_Lip_p}) and the Lipschitz continuity of $v_{\Gamma}$
 we further obtain
 \begin{eqnarray} 
 | B | & \leq & \Bigl( \frac{C}{\sqrt{h}} + L_{H,1} \bigl(1 + \Vert \nabla_{\Gamma} \varphi \Vert_{B(\overline{S_T})}  
 \bigr) \Bigr) \bigl( | x^{n_0}_{i_0} - x^{n_0-1}_{i_0} | + | t^{n_0} - t^{n_0-1} | \bigr) \nonumber   \\
 & & + L_{H,2} | \nabla_{\Gamma} \varphi( x^{n_0}_{i_0},t^{n_0}) - \nabla_{\Gamma}  \varphi(x^{n_0-1}_{i_0},t^{n_0-1}) |  \nonumber \\
& \leq & \frac{C}{\sqrt{h}} \, \tau^{n_0-1} \leq    C \sqrt{h},   \label{Best}
 \end{eqnarray}
where we used \eqref{E:Epsilon_Tau} for the last inequality.
If we insert (\ref{Aest}) and (\ref{Best}) into  (\ref{supervarphi}) and use again (\ref{dervarphi1}), (\ref{dervarphi2})
we obtain
 \begin{equation} \label{Pf_EE:Case1_Ineq_Super}
   -\frac{2}{\sqrt{h}}(t_0-t^{n_0})-\frac{2}{\sqrt{h}}v_\Gamma(x^{n_0}_{i_0},t^{n_0})\cdot(x_0-x^{n_0}_{i_0}) 
    -H\bigl( x^{n_0}_{i_0},t^{n_0},\frac{2}{\sqrt{h}}P_\Gamma(x^{n_0}_{i_0},t^{n_0})(x_0-x^{n_0}_{i_0})\bigr) \leq C
    \sqrt{h}.
  \end{equation}
   We sum up both sides of \eqref{Pf_EE:Case1_Ineq_Sub} and \eqref{Pf_EE:Case1_Ineq_Super} and employ the Lipschitz continuity of $v_{\Gamma}$ as well as (\ref{E:Hamiltonian_Lip_xt}), (\ref{E:Hamiltonian_Lip_p}) to get
  \begin{eqnarray*}
  \lefteqn{  \rho \sqrt{h}  \leq C \sqrt{h} +
   \frac{2}{\sqrt{h}}\{v_\Gamma(x^{n_0}_{i_0},t^{n_0})-v_\Gamma(x_0,t_0)\}\cdot(x_0-x^{n_0}_{i_0}) } \\
  && +H\bigl( x^{n_0}_{i_0},t^{n_0},\frac{2}{\sqrt{h}}P_\Gamma(x^{n_0}_{i_0},t^{n_0})(x_0-x^{n_0}_{i_0})\bigr)
 - H\bigl( x_0,t_0,\frac{2}{\sqrt{h}}P_\Gamma(x_0,t_0)(x_0-x^{n_0}_{i_0})\bigr) \\
 & \leq & C \sqrt{h} +  \frac{ C(|x_0-x^{n_0}_{i_0}|+|t_0-t^{n_0}|)|x_0-x^{n_0}_{i_0}|}{\sqrt{h}}  \\
 & & + L_{H,1}(|x_0-x^{n_0}_{i_0}|+|t_0-t^{n_0}|)\left(1+\frac{2}{\sqrt{h}}|P_\Gamma(x_0,t_0)(x_0-x^{n_0}_{i_0})|\right) \\
   && +\frac{2L_{H,2}}{\sqrt{h}}|P_\Gamma(x_0,t_0)-P_\Gamma(x^{n_0}_{i_0},t^{n_0})||x_0-x^{n_0}_{i_0}| \\
   & \leq & C \sqrt{h} + C \frac{|x_0-x^{n_0}_{i_0}|^2+|t_0-t^{n_0}|^2}{\sqrt{h}} + C \bigl(
   |x_0-x^{n_0}_{i_0}|+|t_0-t^{n_0}| \bigr) \\
   & \leq & C \sqrt{h}
  \end{eqnarray*}
  in view of (\ref{Pf_EE:Const_1}).
 Choosing $\rho>C$ we obtain a contradiction so that this case cannot occur. \\
 \noindent 
  \textbf{Case 2:} $t_0=0$ and $n_0\geq0$. Since $u(x_0,t_0)=u(x_0,0)=u_0(x_0)$ we obtain with the
  help of (\ref{E:Lipschitz_Visc_Sol}), Lemma \ref{L:Nodal_FE_Push} and  (\ref{Pf_EE:Const_1})
  \begin{eqnarray}
  \Psi(x_0,t_0,i_0,n_0) &= & \Psi(x_0,0,i_0,n_0) \leq u(x_0,0) -u^{n_0}_{i_0} = u_0(x_0)-u_0(x^0_{i_0}) +
  u^0_{i_0}-u^{n_0}_{i_0}   \nonumber \\
  & \leq & L_U | x_0 - x^0_{i_0} | +C_4 t^{n_0} \leq C \bigl( | x_0 -x^{n_0}_{i_0} | +
  | x^{n_0}_{i_0} - x^0_{i_0} | \bigr)  + C_4 t^{n_0} \nonumber \\
  & \leq & C \bigl( | x_0 -x^{n_0}_{i_0} | + | t_0 - t^{n_0} | \bigr) \leq  C \sqrt{h}. \label{case2}
  \end{eqnarray}
  \textbf{Case 3:} $t_0\geq0$ and  $n_0=0$. Using once more (\ref{E:Lipschitz_Visc_Sol}) and 
  (\ref{Pf_EE:Const_1}) we derive
   \begin{eqnarray}
  \Psi(x_0,t_0,i_0,n_0) & = & \Psi(x_0,t_0,i_0,0) \leq u(x_0,t_0) -u^{0}_{i_0} = u(x_0,t_0) - u(x^0_{i_0},0)  \nonumber \\
  & \leq & L_U \bigl( | x_0 - x^0_{i_0} | + t_0 \bigr) = L_U  \bigl( | x_0 -x^{n_0}_{i_0} | + | t_0 - t^{n_0} | \bigr)
   \leq  C \sqrt{h}. \label{case3}
  \end{eqnarray}
In conclusion we infer that from (\ref{maxbound}), (\ref{case2}), (\ref{case3}) and the fact that
Case 1 cannot occur that
\begin{displaymath}
 \max_{1\leq i\leq M,\,0\leq n\leq N} (u(x_i^n,t^n)-u_i^n) \leq C \sqrt{h}.
\end{displaymath}
In an analogous way we bound $ \max_{1\leq i\leq M,\,0\leq n\leq N} (u^n_i-u(x_i^n,t^n))$ which
completes the proof of the theorem.
\end{proof}

  \section{Numerics} \label{S:Numerics}
In this section we present some numerical results.
In order to implement the scheme it is necessary to triangulate the initial surface and then evolve the vertices using the surface  material velocity.  Vertex evolution  would typically be done  by time stepping with a sufficiently accurate  ordinary differential equation solver using the known material velocity.   The scheme has been designed to allow non-acute triangulations which may be the consequence of an evolution from an initially acute triangulation. Note that for coupled systems the evolution of the surface may depend on the solution of the surface PDE.  Also it may be of interest to solve equations on unstructured evolving  triangulations arising from the data analysis of experimental  observations, c.f. \cite {BreDuEll16}.  At each time step we allow a variable  $\varepsilon_i^n$ and a variable $\tau^n$. Note that the scheme is also implementable with these parameters being constant and still satisfying the constraints  (\ref{E:Epsilon_Tau}) provided one has good estimates of the requisite mesh sizes. The discrete Hamiltonian  (\ref{E:Numerical_Hamiltonian}) requires mesh computations  at each vertex using elementary trigonometric formulae so the mesh parameters are readily available.  In the simulations we present %Examples \ref{ExOne} and \ref{ExTwo} 
the surfaces are sufficiently simple that the vertices of the evolving triangulations are known exactly. % we begin by investigating the experimental order of convergence (eoc) 
%and then we show numerical simulations from the motivating example in Section \ref{s:motivating}.  

%\begin{example}\label{ExOne}
%To begin we consider the model problem of an expanding sphere. Let $u=e^{-t}x_1x_2x_3$ on an expanding sphere $\Gamma(t)$  with $\Gamma(0)=S^1$ and velocity $\displaystyle{v_\Gamma=\frac{x}{|x|}}$. It follows that the radius of the sphere is $1+t$ and the the positions of vertices are easily calculated by formula. Since the motion is in the normal direction it follows that 
%$H(x,t,p)= F(x,t) \, | p | +\beta(x,t) \cdot p$  \textcolor{red}{where  $F=$ and $\beta=$. An initial triangulation is displayed in  Figure for an example of $h$.}
%We investigate the experimental order of convergence $EOC$ which is the ratio of errors for successive reduction of the largest triangle edge, $h$, of the initial triangulation.  In Table \ref{table:eoc} we display the values of 
%$$\mathcal{E}_h= \max_{1\leq i\leq M,\,0\leq n\leq N} |u^n_i-u(x_i^n,t^n)|$$ 
%together with the corresponding $EOC$s. 
%We chose  the mesh parameters $\epsilon_i^n =C \tau^n=$ according to with the choices 

\begin{example}
To begin we consider model problems for which we have explicit solutions. To achieve this we  consider  an expanding 
sphere $\Gamma(t)$  with $\Gamma(0)=S^1$ and velocity $v_\Gamma=x/|x|$. 
It follows that the flow map, (\ref{E:Flow_Map}),  is given by $\Phi(X,t):=(1+t)X$ so that the radius of the sphere is $R(t)=1+t$ and the  positions of vertices are easily calculated by formula. 

Note that for a given function $g(x,t), x\in \mathbb R^3, t\ge 0$ on $\Gamma(t)$ 
$$
|\nabla_{\Gamma} g|^2 = |\nabla g|^2 - \frac{(\nabla g\cdot x)^2}{R^2}\mbox~~~\mbox{and}~~~\partial^\bullet g=g_t +\frac{x\cdot\nabla g}{R}.
$$
Using this $g$ we set
$$
H(x,t,p)=  \Bigl (- | p |+  \bigl(|\nabla g(x,t)|^2 - \frac{(\nabla g(x,t)\cdot x)^2}{R(t)^2}\bigr)^{1/2} \Bigr)  -
\bigl(g_t(x,t) +\frac{x \cdot \nabla g(x,t)}{R(t)} \bigr).
$$ 

It follows that $u(x,t):=g(x,t), t \ge 0, x\in \Gamma(t)$ solves (\ref{E:HJ_Equation}).  

We present two examples, in the first we set $g=e^{-0.5t}x_1x_2x_3$ and in the second we set $g=10\sin(t)+x_1x_2x_3t$. For each example we use two initial triangulations, one with a non-acute mesh and one with an acute one, the associated triangulations at $t=0.5$ are displayed in Figure \ref{f:init_tri}.
% In all simulations we set $\tau^n=0.005\min_{i,j}|E_j^{n,i}|$. \\

We investigate the experimental order of convergence, EOC, which is the ratio of errors for successive reduction of the largest triangle edge, $h$, of the initial triangulations. The time step is chosen to be  $\tau^n=0.005\min_{i,j}|E_j^{n,i}|$. 
%In the tables 
In the results we display the values of 
$$\mathcal{E} = \max_{1\leq i\leq M,\,0\leq n\leq N} |u^n_i-u(x_i^n,t^n)|,$$ 
together with the corresponding EOCs for the time interval  $t\in (0,0.5)$. 

The EOCs for $u=e^{-0.5t}x_1x_2x_3$, with $\varepsilon_i^n=C_1\max_j h_{T_j^{n,i}}$, for $C_1=0.5,0.2,0.1$, are displayed in Tables 1 and 2, with Table 1 corresponding to the non-acute triangulation and Table 2 corresponding to the acute triangulation. From Tables 1 and 2, for $C_1=0.5$, we see convergence of the solution, with in the case of the acute triangulation, an EOC that is approaching $1$. However once $C_1$ is reduced to $0.1$ the convergence is lost for the non-acute triangulation and the EOCs are much reduced for the acute triangulation.

We see similar behaviour for  the convergence of the solution in Tables 3 and 4 where the corresponding results for $u=10\sin(t)+x_1x_2x_3t$ are displayed, again with $\varepsilon_i^n=C_1\max_j h_{T_j^{n,i}}$, for $C_1=0.5,0.2,0.1$.

%We chose  the mesh parameters $\epsilon_i^n =C \tau^n=$ according to with the choices 

%Triangulations at $t=0.5$ are displayed in Figure \ref{f:init_tri}. % for an example of $h$.

\begin{figure}[h]
\centering
\includegraphics[scale=0.17,clip]{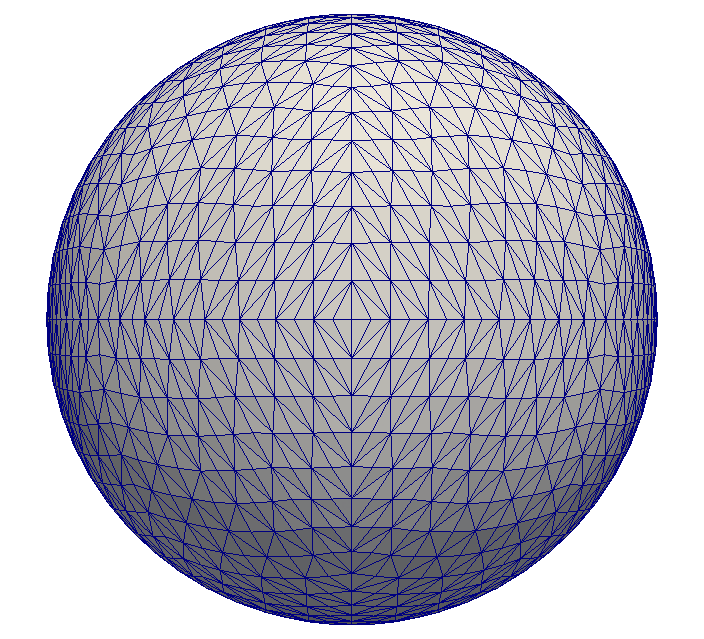} \hspace{1cm}
\includegraphics[scale=0.17,clip]{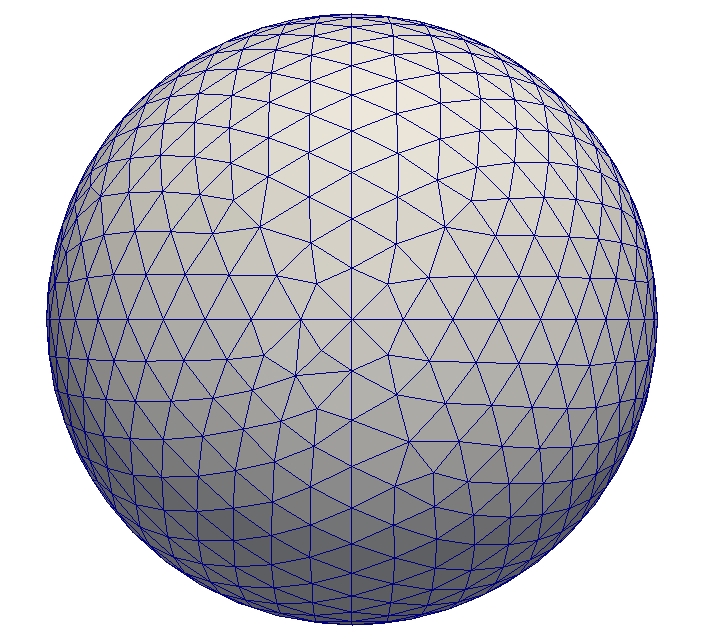}
\caption{Triangulations at $t=0.5$ with $\max_{j}h_{T_j^{n,i}}=0.2164$  (left) and $\max_{j}h_{T_j^{n,i}}=0.2443$  (right).}
\label{f:init_tri}
\end{figure}

\begin{table}[!h]
 \begin{center}
 \begin{tabular}{ |c|c|c|c|c|c|c| }
 \hline
$C_1$ & \multicolumn{2}{c|}{$0.5$} & \multicolumn{2}{c|}{$0.2$} & \multicolumn{2}{c|}{$0.1$} \\
\hline
$h_{max}$ & $\mathcal{E}$ & EOC& $\mathcal{E}$ & EOC& $\mathcal{E}$ & EOC\\
\hline
0.8359539 & 0.3746441 & -   & 0.1979334 & -    & 0.1442588 & -     \\ 
0.2164580 & 0.1508848 & 0.6731 & 0.0728049 & 0.7402  & 0.0583846 & 0.6695\\ 
0.0542420 & 0.0495660 & 0.8044  & 0.0270853 & 0.7145& 0.0744852 & -0.1760   \\ 
0.0135628 & 0.0225146 & 0.5693  & 0.0173534 & 0.3212 & 0.0601878 & 0.1538 \\ 
\hline
\end{tabular}
\caption{Non-acute triangulation, $u=e^{-0.5t}x_1x_2x_3$,  $\tau^n=0.005\min_{i,j}|E_j^{n,i}|$, $\varepsilon^n_i=C_1\max_{j}h_{T_j^{n,i}}$}
\label{table:eoc}
\end{center}
\end{table}

\begin{table}[!h]
 \begin{center}
 \begin{tabular}{ |c|c|c|c|c|c|c| }
 \hline
$C_1$ & \multicolumn{2}{c|}{$0.5$} & \multicolumn{2}{c|}{$0.2$} & \multicolumn{2}{c|}{$0.1$} \\
\hline
$h_{max}$ & $\mathcal{E}$ & EOC& $\mathcal{E}$ & EOC& $\mathcal{E}$ & EOC\\
\hline
1.1218880 & 0.3500829 & -   & 0.1834464 & -     & 0.1946865 & -    \\ 
0.2444325 & 0.0980209 & 0.8354 & 0.0453545 & 0.9170 & 0.0627731 & 0.7428 \\ 
0.0665852 & 0.0299586 & 0.9115& 0.0170005 & 0.7546   & 0.0499775 & 0.1753  \\ 
0.0173820 & 0.0083878 & 0.9479   & 0.0058230 & 0.7978 & 0.0214375 & 0.6302  \\ 
\hline
\end{tabular}
\caption{Acute triangulation, $u=e^{-0.5t}x_1x_2x_3$,  $\tau^n=0.005\min_{i,j}|E_j^{n,i}|$, $\varepsilon^n_i=C_1\max_{j}h_{T_j^{n,i}}$}
\label{table:eoc}
\end{center}
\end{table}

\begin{table}[!h]
 \begin{center}
 \begin{tabular}{ |c|c|c|c|c|c|c| }
 \hline
$C_1$ & \multicolumn{2}{c|}{$0.5$} & \multicolumn{2}{c|}{$0.2$} & \multicolumn{2}{c|}{$0.1$} \\
\hline
$h_{max}$ & $\mathcal{E}$ & EOC& $\mathcal{E}$ & EOC& $\mathcal{E}$ & EOC\\
\hline
0.8359539 & 0.1361987 & -   & 0.0638936 & -  & 0.0515144 & -    \\ 
0.2164580 & 0.0548873 & 0.6726 & 0.0273127 & 0.6290 & 0.0184992 & 0.7580  \\ 
0.0542420 & 0.0204888 & 0.7120  & 0.0118116 & 0.6057 & 0.0166502 & 0.0761  \\ 
0.0135628 & 0.0100919 & 0.5109  & 0.0069725 & 0.3803& 0.0601878 & -0.9271 \\ 
\hline
\end{tabular}
\caption{Non-acute triangulation, $u=10\sin(t)+x_1x_2x_3t$, $\tau^n=0.005\min_{i,j}|E_j^{n,i}|$, $\varepsilon^n_i=C_1\max_{j}h_{T_j^{n,i}}$}
\label{table:eoc}
\end{center}
\end{table}

\begin{table}[!h]
 \begin{center}
 \begin{tabular}{ |c|c|c|c|c|c|c| }
 \hline
$C_1$ & \multicolumn{2}{c|}{$0.5$} & \multicolumn{2}{c|}{$0.2$} & \multicolumn{2}{c|}{$0.1$} \\
\hline
$h_{max}$ & $\mathcal{E}$ & EOC& $\mathcal{E}$ & EOC& $\mathcal{E}$ & EOC\\
\hline
1.1218880 & 0.1295989 & -   & 0.0843471 & -   & 0.0714247 & -         \\ 
0.2444325 & 0.0319010 & 0.9199 & 0.0160280 & 1.0898& 0.0193385 & 0.8574  \\ 
0.0665852 & 0.0103778 & 0.8635 & 0.0045316 & 0.9714   & 0.0138817 & 0.2549  \\ 
0.0173820 & 0.0030784 & 0.9049  & 0.0018737 & 0.6576 & 0.0061684 & 0.6039  \\ 
\hline
\end{tabular}
\caption{Acute triangulation, $u=10\sin(t)+x_1x_2x_3t$, $\tau^n=0.005\min_{i,j}|E_j^{n,i}|$, $\varepsilon^n_i=C_1\max_{j}h_{T_j^{n,i}}$}
\label{table:eoc}
\end{center}
\end{table}
\end{example}

\begin{example}
We conclude with a simulation of the evolution of curves on a smoothly evolving surface, as in the motivating example in Section \ref{s:motivating}. In particular we consider the zero level set of a function as defining the curve.
We set $\Gamma(0):=\{x\in \mathbb{R}^3| x_1^2+x_2^2+2x_3^2(x_3^2-\frac{199}{200})=0.01\}$, $F=1+4x_1^2$, $\beta = (1,0.1,-0.8)^T$ 
and $u(0)=(x_3+0.3)(x_3-0.1)-0.3$, such that $\gamma(0)$ consists of two circular curves  
%$x_1^2+x_2^2=0.01+2(0.1\sqrt{0.34})(0.35-0.1\sqrt{0.34}-0.645)$ and 
%$x_1^2+x_2^2=0.01-2(0.35+0.1\sqrt{0.34})(0.1\sqrt{0.34}-0.645)$, 
lying in the planes $x_3=-\sqrt{0.34}-0.1$ and $x_3=\sqrt{0.34}-0.1$. 
% $x_1^2+x_2^2=0.01-0.8(0.4-0.995)$, lying in the planes $x_3=\pm \sqrt{0.4}$. 
The velocity of the $j$-th node of the triangulation is taken to be $v_{\Gamma,j}=\pi(\sin(2\pi t)X_1^j(0),\sin(2\pi t)X_2^j(0),0.8\sin(4\pi t)X_3^j(0))$, where 
$X_i^j(0)$, $i=1,2,3$, denotes the $i$-th coordinate of the $j$-th node of the initial triangulation with $\mathbf{X}^j(0)\in \Gamma(0)$. 
The results are displayed in Figure \ref{f:db} in which the evolving curves $\gamma(t)$ are approximated by the zero level line of $u$ which is depicted by a white line. 
In this simulation we set $\tau^n = 0.01\min_{i,j}|E_j^{n,i}|$ and $\varepsilon_i^n =0.5\max_{j}h_{T_j^{n,i}}$. 

\begin{figure}[h]
\centering
\includegraphics[scale=0.24,clip]{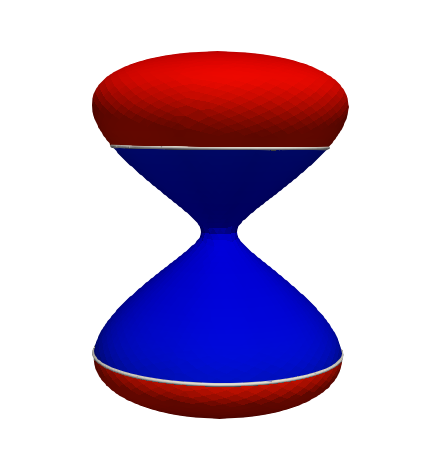} 
\includegraphics[scale=0.24,clip]{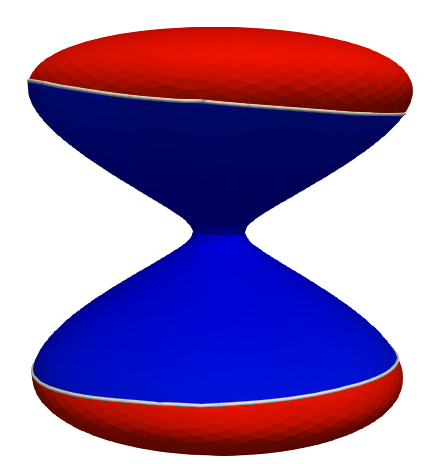}  \hspace{0mm}
\includegraphics[scale=0.24,clip]{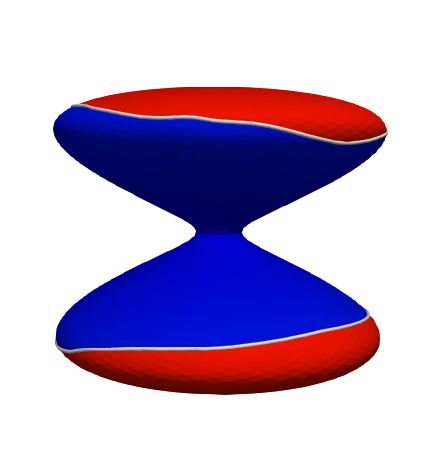}\hspace{0mm}
\includegraphics[scale=0.24,clip]{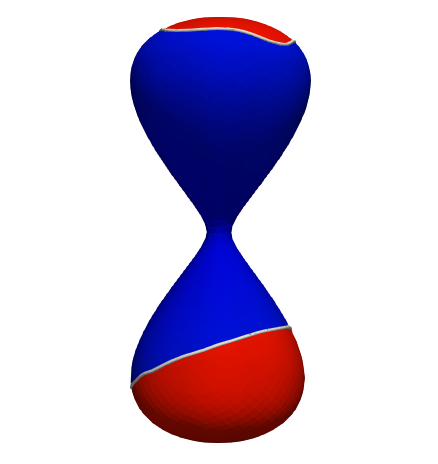}  
\caption{$t=0.0,0.2,0.4,0.6$, $F=1+4x_1^2$, $\beta = (1,0.1,-0.8)^T$.}
\label{f:db}
\end{figure}
\end{example}

\bibliographystyle{amsplain}
\providecommand{\bysame}{\leavevmode\hbox to3em{\hrulefill}\thinspace}
\providecommand{\MR}{\relax\ifhmode\unskip\space\fi MR }
\providecommand{\MRhref}[2]{%
  \href{http://www.ams.org/mathscinet-getitem?mr=#1}{#2}
}
\providecommand{\href}[2]{#2}

%\bibliography{HJ_Ref}

\end{document}